\documentclass[11pt,reqno]{amsart}

\usepackage{amsfonts, amsthm, amsmath, amssymb}
\allowdisplaybreaks[4]

\usepackage{tikz} 
\usetikzlibrary{arrows, arrows.meta}

\usepackage{float} 
\usepackage{multirow} 
\usepackage{supertabular} 
\usepackage{framed}

\usepackage{footnote}
\makesavenoteenv{minpage}   
\makesavenoteenv{itemize}  

\makeatletter

\newcommand{\Rmnum}[1]{\expandafter\@slowromancap\romannumeral #1@}
\makeatother

\usepackage{enumerate} 
\usepackage[inline]{enumitem} 

\usepackage[pagebackref]{hyperref} 
\hypersetup{colorlinks=true} 

\usepackage{mathabx}

\usepackage{extpfeil} 

\usepackage{mathtools} 

\numberwithin{equation}{section}
\theoremstyle{plain}

\newtheorem{theorem}{Theorem}[section]

\newtheorem{claim}[theorem]{Claim}
\newtheorem{proposition}[theorem]{Proposition}

\theoremstyle{definition}
\newtheorem{Def}[theorem]{Definition}
\newtheorem{example}[theorem]{Example}
\newtheorem{conj}[theorem]{Conjecture}
\newtheorem{remark}[theorem]{Remark}
\newtheorem{?}[theorem]{Problem}

\newcommand{\SF}[1]{\color{magenta}#1}

\newcommand{\bN}{\mathbb{N}}
\newcommand{\bR}{\mathbb{R}}
\newcommand{\bx}{\mathbf{x}}
\newcommand{\A}{\mathcal{A}}
\newcommand{\B}{\mathcal{B}}
\newcommand{\D}{\mathcal{D}}
\newcommand{\cE}{\mathcal{E}}
\newcommand{\fD}{\mathfrak{D}}

\renewcommand{\S}{\mathfrak{S}}

\newcommand{\fG}{\mathfrak{G}}
\newcommand{\cH}{\mathcal{H}}
\newcommand{\cK}{\mathcal{K}}
\newcommand{\cR}{\mathcal{R}}
\newcommand{\cX}{\mathcal{X}}
\newcommand{\GI}{\fG^{\mathrm{I}}}
\newcommand{\tGI}{\widehat{\fG}^{\mathrm{I}}}
\newcommand{\cGI}{\widecheck{\fG}^{\mathrm{I}}}
\newcommand{\GII}{\fG^{\mathrm{II}}}
\newcommand{\GIII}{\fG^{\mathrm{III}}}
\newcommand{\GIV}{\fG^{\mathrm{IV}}}
\newcommand{\Gi}{G^{\mathrm{I}}}

\newcommand{\LI}{\Lambda^{\mathrm{I}}}
\newcommand{\LII}{\Lambda^{\mathrm{II}}}
\newcommand{\LIII}{\Lambda^{\mathrm{III}}}
\newcommand{\LIV}{\Lambda^{\mathrm{IV}}}
\newcommand{\proj}{\mathrm{proj}}
\newcommand{\DIII}{\D^{\mathrm{III}}}
\newcommand{\tD}{\widehat{\D}^{\mathrm{III}}}
\newcommand{\cD}{\widecheck{\D}^{\mathrm{III}}}

\def\oO{\mathrm{oO}}
\def\oE{\mathrm{oE}}
\def\eE{\mathrm{eE}}
\def\eO{\mathrm{eO}}
\def\Oo{\mathrm{Oo}}
\def\Eo{\mathrm{Eo}}
\def\Ee{\mathrm{Ee}}
\def\Oe{\mathrm{Oe}}
\def\Re{\mathrm{Re}}
\def\re{\mathrm{red}}
\def\CO{\mathcal{CO}}
\def\Co{\mathrm{CO}}
\def\DC{\mathcal{DC}}
\def\EC{\mathcal{EC}}
\def\PLY{\Psi_{\mathrm{LY}}}
\def\Lrmin{\mathrm{Lrmin}}
\def\Cmin{\mathrm{Cmin}}

\begin{document}

\title[Parity patterns meet Genocchi numbers, I.]{Parity patterns meet Genocchi numbers, I: four labelings and three bijections}

\author[Q. Yuan]{Quan Yuan}
\address[Quan Yuan]{College of Mathematics and Statistics, Chongqing University, Chongqing 401331, P.R. China}
\email{springyuan0929@163.com}

\author[Q. Fang]{Qi Fang}
\address[Qi Fang]{College of Mathematics and Statistics, Chongqing University, Chongqing 401331, P.R. China}
\email{qifangpapers@163.com}

\author[S. Fu]{Shishuo Fu} 
\address[Shishuo Fu]{College of Mathematics and Statistics, Chongqing University \& Key Laboratory of Nonlinear Analysis and its Applications (Chongqing University), Ministry of Education, Chongqing 401331, P.R. China}
\email{fsshuo@cqu.edu.cn}

\author[H. Li]{Haijun Li}
\address[Haijun Li]{College of Mathematics and Statistics, Chongqing University, Chongqing 401331, P.R. China}
\email{lihaijun@cqu.edu.cn}

\date{\today}

\begin{abstract}
Hetyei introduced in 2019 the homogenized Linial arrangement and showed that its regions are counted by the median Genocchi numbers. In the course of devising a different proof of Hetyei's result, Lazar and Wachs considered another hyperplane arrangement that is associated with certain bipartite graph called Ferrers graph. We bijectively label the regions of this latter arrangement with permutations whose ascents are subject to a parity restriction. This labeling not only establishes the equivalence between two enumerative results due to Hetyei and Lazar-Wachs, repectively, but also motivates us to derive and investigate a Seidel-like triangle that interweaves Genocchi numbers of both kinds. 

Applying similar ideas, we introduce three more variants of permutations with analogous parity restrictions. We provide labelings for regions of the aforementioned arrangement using these three sets of restricted permutations as well. Furthermore, bijections from our first permutation model to two previously known permutation models are established.
\end{abstract}


\maketitle


\section{Introduction}
The (signless) {\it Genocchi numbers} $\{g_n\}_{n\ge 1}=\{1,1,3,17,155,2073,38227,\ldots\}$ \cite[A110501]{slo} and its close relative the (signless) {\it median Genocchi numbers} $\{h_n\}_{n\ge 0}=\{1,2,8,56,608,9440,\ldots\}$ \cite[A005439]{slo} are two sequences that bear number theoretical, combinatorial, and most recently geometrical interest. Due to their direct relation with the even-indexed Bernoulli numbers $B_{2n}$, namely $g_n=2(1-2^{2n})(-1)^nB_{2n}$, the Genocchi numbers have a neat exponential generating function given by \cite{dum74}
$$
\sum_{n\ge 1}g_n\frac{x^{2n}}{(2n)!}=x\tan\frac{x}{2}.
$$ 
The two kinds of Genocchi numbers are related via the following identity that holds for all $n\in\bN$.
\begin{align}\label{id:two Genocchi}
h_n=\sum_{j=0}^{\lfloor\frac{n}{2}\rfloor}(-1)^j\binom{n+1}{2j+1}g_{n+1-j}.
\end{align}

Basing on Gandhi's generation of Genocchi numbers, Dumont~\cite{dum74} discovered in 1974 the first ever combinatorial interpretations of Genocchi numbers in terms of certain permutations subject to parity conditions on their descents and ascents (see Definition~\ref{def:five perm}). Since then, a wealth of combinatorial models whose enumerations yield either Genocchi numbers or median Genocchi numbers were found; see for example~\cite{DV80,vie81,DR94,BJS10,fei11,big14,het19,LW22,HJO23,FLS25,DLP25} and the references contained in the aforementioned two OEIS entries. 

A quite recent development stems from geometric considerations. Firstly, Hetyei~\cite{het19} showed in 2019, applying Athanasiadis's~\cite{ath96} finite field method of counting regions created by hyperplane arrangements, that certain tournaments are enumerated by the median Genocchi numbers (see Theorem~\ref{thm:hetyei}). In a follow up paper, Lazar and Wachs~\cite{LW22} made direct use of Zaslavsky's formula to relate the intersection lattice associated with a hyperplane arrangement to the number of its regions. Consequently, they were able to reprove and refine Hetyei's result. As an interesting byproduct of their work, Lazar and Wachs~\cite{LW22} introduced a new witness of median Genocchi numbers. In order to state their results, we recall some definitions. We write a permutation $\sigma$ of $[n]:=\{1,2,\ldots,n\}$ using the one-line notation $\sigma=\sigma_1\sigma_2\cdots\sigma_n$ where $\sigma_i=\sigma(i)$ for $1\le i\le n$. Following \cite{LW22}, we call the pair $(i,\sigma_i)$ a {\it drop} whenever $i>\sigma_i$. Note that it goes by other names such as ``anti-excedance'' in the literature; see for instance \cite{SZ22}. In particular, an {\it even-odd drop} refers to a drop $(i,\sigma_i)$ with $i$ being even and $\sigma_i$ being odd. Let $\S_n$ denote the set of permutations of $[n]$. Lazar and Wachs derived the following new interpretation for the median Genocchi numbers.
\begin{theorem}[{\cite[Coro.~6.2]{LW22}}]\label{thm:Lazar-Wachs}
For all $n\ge 1$, we have 
$$h_n=|\{\sigma\in\S_{2n}: \text{$\sigma$ has only even-odd drops}\}|.$$
\end{theorem}
They continued to make the following conjecture.
\begin{conj}[{\cite[Conj.~6.4]{LW22}}]\label{conj:Lazar-Wachs}
For all $n\ge 1$, the $n$-th Genocchi number $g_n$ is equal to the number of cycles on $[2n]$ with only even-odd drops.
\end{conj}

It is worth noting that these observations made by Lazar and Wachs are somewhat peculiar since usually one would expect a combinatorial model counted by the Genocchi numbers to encompass those counted by the median Genocchi numbers as a subset, but the contrast between Theorem~\ref{thm:Lazar-Wachs} and Conjecture~\ref{conj:Lazar-Wachs} above is exactly the other way around. Conjecture~\ref{conj:Lazar-Wachs} aroused a lot of interest. Within two years from the time when the arXiv version of \cite{LW22} was released, there are at least three independent proofs of Conjecture~\ref{conj:Lazar-Wachs} that we are aware of. The proof by Lin and Yan~\cite{LY22} is bijective and extends to the multiset setting. The proof by Pan and Zeng~\cite{PZ23} relies on an explicit J-fraction formula of an ordinary generating function that involves eight statistics. Finally, the proof by Chern~\cite{che23} employs a generating tree and partial differential equations and could be applied to enumerate another class of permutations subject to a similar parity condition.

Let $\cH_{2n-1}$ denote the homogenized Linial arrangment as introduced by Hetyei~\cite{het19} (see section~\ref{sec:pre} for the precise definition). One crucial step in Lazar-Wachs's derivation was to characterize the intersection lattice associated with $\cH_{2n-1}$ as a bond lattice associated with certain bipartite graph, then calculate its M\"obius function by counting non-broken-circuit (NBC) sets. Certain graphical arrangement $\cK_{2n}$ was introduced as an intermediate object in \cite[Thm.~3.2]{LW22}. We are able to construct a direct bijection between the regions of $\cK_{2n}$ and a set of restricted permutations in $\S_{2n}$ that we denote as $\GI_{2n}$. Its definition given below involves the notion of {\it parity patterns}, which will be introduced in Section~\ref{sec:labeling}. 

\begin{Def}\label{def:G-I}
	For every $n\ge 1$, we let $\GI_{2n} :=\{ \pi \in \S_{2n}: \pi \text{ avoids parity patterns }\eE,\eO,\oO \}$. Alternatively, $\pi \in \GI_{2n}$ if and only if every ascent $\pi_i<\pi_{i+1}$ of $\pi$ satisfies that $\pi_i$ is odd and $\pi_{i+1}$ is even.
\end{Def}

For instance, $\GI_2=\{12,21\}$ and $\GI_4=\{1432,2143,3142,3214,3412,3421,4312,4321\}$. 

The following labeling is one out of four distinct labelings for the regions of $\cK_{2n}$, and could be viewed as the first main result of our paper. Given a hyperplane arrangement $\A$, denote by $\cR(\A)$ the set of regions carved out by all the hyperplanes in $\A$ and use $r(\A)$ to denote its cardinality.

\begin{theorem}\label{thm:label-I}
For any $n\ge 1$, the regions of $\cK_{2n}$ are bijectively labeled, via the mapping $\Lambda_{\mathrm{I}}$, by the permutations in $\GI_{2n}$. 
\end{theorem}

Going from a sequence of numbers to a triangular number array is often a good way of making new inroads on the study of the original sequence. For Genocchi numbers of the two kinds, one such triangular array is known as the {\it Seidel triangle}~\cite{sei77} and is well studied; see \cite{ES00} and \cite{BJS10} for related works. The first few rows of the Seidel triangle are shown in Table~\ref{Tab:Seidel triangle} below, wherein the numbers on the rightmost diagonal are the Genocchi numbers, while the leftmost column consists of the Genocchi medians, with each number occurring twice. The alternating arrows on the margin are to indicate the generation of a specific number by adding together the number preceding it in the same row (using the arrow for the order) and the number above it in the same column. For example, the entry at $(n,k)=(7,3)$ is $17=14+3$. This way of creating a number array is typically referred to as {\it boustrophedon}~\cite{PZ23}.

\begin{table}[ht]
    \centering
    \begin{tabular}{c|cccccccc}
    	$n \backslash k$ & & 1 & 2 & 3 & 4 & 5  \\ \hline
         1 & & 1 & & & & \\ 
         2 & & 1 & $\leftarrow$ & & & \\ 
    3&$\rightarrow$& 1 & 1 & & &  \\ 
         4 & & 2 & 1 & $\leftarrow$ & &  \\ 
    5&$\rightarrow$& 2 & 3 & 3 & &  \\ 
         6 & & 8 & 6 & 3 & $\leftarrow$ & \\
    7&$\rightarrow$& 8 & 14 & 17 & 17 &  \\ 
         8 & & 56 & 48 & 34 & 17 & $\leftarrow$ \\
    9&$\rightarrow$& 56 & 104 & 138 & 155 & 155 \\
         10 & & 608 & 552 & 448 & 310 & 155 & $\leftarrow$ \\
    \end{tabular}
    \medskip
    \caption{The first ten rows of the Seidel triangle} \label{Tab:Seidel triangle}
\end{table}

More precisely, the entry in the $n$-th row and $k$-th column of the Seidel triangle, denoted $S_{n,k}$, can be generated according to the following rules:
\begin{align}
	& S_{1,1} = 1,\quad S_{n,k} = 0 \text{ if $k\not\in [1,\lceil\frac{n}{2}\rceil]$,}\nonumber \\
	& S_{2n,k} = S_{2n-1,k}+S_{2n,k+1},\quad S_{2n+1,k} = S_{2n,k}+S_{2n+1,k-1}.\label{Seidel rule}
\end{align}
In particular, for every $n\ge 1$ we have
\begin{align}\label{id:Seidel two Genocchi}
& S_{2n,1} = h_{n-1}, \qquad S_{2n,n}=g_n.
\end{align}

Our second main result is to refine the counting of $\GI_{2n}$ by the first letter of the permutation and derive the following relations for the first Seidel-like triangle thus obtained; see Table~\ref{Tab:GI} for the first six rows of this triangle. Let us denote $\Gi_{2n}:=|\GI_{2n}|$, $\GI_{2n,k}:=\{\sigma\in\GI_{2n}: \sigma_1=k\}$, and $\Gi_{2n,k}:=|\GI_{2n,k}|$ for all $n\ge 1$ and $1\le k\le 2n$.

\begin{theorem} \label{thm:GIS}
For all $n\ge k\ge 1$, we have
\begin{align}
\Gi_{2n,1} &= \Gi_{2n,2},\label{id:1=2}\\
\Gi_{2n,3} &= 2(\Gi_{2n,1}+\Gi_{2n-2,1}),\label{E_3}\\
\Gi_{2n,2n} &= \sum_{i=1}^{2n-2} \Gi_{2n-2,i}=\Gi_{2n-2},\label{En2n=sum of above line} \\
2\Gi_{2n,2n-1} &= \sum_{i=1}^{2n} \Gi_{2n,i},\label{eq:GI_2n,2n-1}\\
\Gi_{2n,2k} &= S_{2n,n+1-k}. \label{id:GI and S}
\end{align}
\end{theorem}

\begin{theorem} \label{En-even}
	For all $n\ge 2,n>k\ge 0$, we have 
\begin{align}
\Gi_{2n,2k+2}=\Gi_{2n,2k}+\sum\limits_{i=k}^{n-1} \Gi_{2n-2,2i}.\label{id:GI even rec}
\end{align}
\end{theorem}

\begin{table}[ht]
    \renewcommand{\arraystretch}{1.1}
    \centering
    \begin{tabular}{c|cccccccccccccc}
        $n \backslash k$ & 1 & 2 & 3 & 4 & 5 & 6 & 7 & 8 & 9 & 10 & 11 & 12 \\ \hline
        1 & 1 &1  \\ 
        2 & 1 &1 & 4 & 2  \\ 
        3 & 3 &3 & 8 & 6 & 28 & 8  \\ 
        4 & 17 &17 & 40 & 34 & 92 & 48 & 304 & 56  \\ 
        5 & 155 &155 & 344 & 310 & 676 & 448 & 1472 & 552 & 4720 & 608  \\ 
        6 & 2073 & 2073 & 4456 & 4146 & 8060 & 6064 & 14848 & 7672 & 31472 & 8832 & 99136 & 9440 \\
    \end{tabular}
    \smallskip
    \caption{The Seidel-like triangle refining the set $\GI_{2n}$} \label{Tab:GI}
\end{table}

The rest of the paper is organized as follows. We collect definitions and some preliminary results in Section~\ref{sec:pre} to prepare ourselves for the proof of Theorem~\ref{thm:label-I} in Section~\ref{sec:labeling}, where three more sets of parity pattern avoiding permutations, namely $\GII_{2n}$, $\GIII_{2n}$, and $\GIV_{2n}$ are introduced, together with three labelings connecting each of them with the regions of $\cK_{2n}$. Section~\ref{sec:labeling} ends with a proof of Theorem~\ref{thm:GIS}, relying on the insights we gained from the first labeling $\LI$. Section~\ref{sec:bijs} features three bijections, the first of which is constructed recursively and it amounts to give a second proof of \eqref{id:GI and S}. The remaining two bijections are utilized to link our first set of permutations $\GI_{2n}$ to two previously considered permutation models in the literature. 


\section{Preliminaries}\label{sec:pre}
We collect in this section definitions and results that will be utilized in later sections. They are mainly for hyperplane arrangment and various kinds of restricted permutations.

A {\it (real) hyperplane arrangement} $\A\subseteq \bR^n$ is a finite collection of affine codimension one subspaces in $\bR^n$. Clearly the complement $\bR^n\backslash \A$ is disconnected, and we denote $r(\A)$ the number of connected components (called {\it regions}) of $\bR^n\setminus \A$. When the hyperplane arrangment in question is defined for every dimension, say $\A_n$ stands for the corresponding arrangment in dimension $n$, the natural question then arises as to whether or not the counting sequence $\{r(\A_n)\}_{n\ge 1}$ agrees with a known sequence in the literature. Better yet, bijective combinatorists would like to find certain well known combinatorial model with the correct cardinalities (i.e., given by $\{r(\A_n)\}_{n\ge 1}$) and understand how one can label each region using the objects from this model. For instance $\S_n$, the set of $n$-permutations, is a natural choice for labeling the regions of the {\it braid arrangement} in $\bR^n$:
$$\B_n:=\{x_i-x_j=0 : 1\le i < j\le n\},$$
and the celebrated Pak-Stanley labeling~\cite{sta96} for the regions of the {\it Shi arrangement} uses parking functions. For further interesting aspects of the hyperplane arrangements, the reader is referred to Stanley's notes~\cite{sta07} and book~\cite[Chapter~3.11]{sta11}.

Motivated by the work of Postnikov and Stanley~\cite{PS00}, wherein the {\it Linial arrangement} was bijectively labeled by semiacyclic tournaments, Hetyei~\cite{het19} considered a wider class called alternation acyclic tournaments and introduced the following {\it homogenized Linial arrangement}
$$\cH_{2n-1}:=\{x_i-x_j=y_i : 1\le i < j \le n+1\},$$
which is a hyperplane arrangement in 
$$\{(x_1,\ldots,x_{n+1},y_1,\ldots,y_n): (x_i,y_j)\in\bR^2\text{ for }(i,j)\in[n+1]\times[n]\}=\bR^{2n+1},$$
and he enumerated the regions of $\cH_{2n-1}$, relying on a formula derived in \cite[Corollary~5.5]{CKRT10} that connects the Genocchi medians with the so-called Legendre-Stirling numbers.
\begin{theorem}[{\cite[Thms. 4.3 \& 4.4]{het19}}]\label{thm:hetyei}
For all $n\ge 1$, we have $r(\cH_{2n-1})=h_n$.
\end{theorem}

To produce an alternative proof of Theorem~\ref{thm:hetyei}, Lazar and Wachs~\cite[Thm.~3.2]{LW22} introduced another hyperplane arrangement in $\bR^{2n+1}=\{(x_1,\ldots,x_{2n+1}):x_i\in\bR\}$, namely
$$\cK_{2n}:=\{x_{2i-1}-x_{2j}=0 : 1\le i \le j \le n\},$$
and they established the following equivalence between $\cK_{2n}$ and $\cH_{2n-1}$.
\begin{proposition}\label{prop:K2n}
There exists a vector space isomorphism $\psi:\bR^{2n+1}\to\bR^{2n+1}$ that takes $\cK_{2n}$ to $\cH_{2n-1}$. Consequently, the regions of $\cK_{2n}$ corresponds bijectively to the regions of $\cH_{2n-1}$ and we have in particular $r(\cK_{2n})=r(\cH_{2n-1})$.
\end{proposition}

One advantage of studying $\cK_{2n}$ over $\cH_{2n-1}$ is that the intersection lattice of $\cK_{2n}$ can be identified with the bond lattice of certain bipartite graph $\Gamma_{2n}$, which has $\{1,3,\ldots,2n-1\}\sqcup\{2,4,\ldots,2n\}$ as its set of vertices and $\{(2i-1,2j):1\le i\le j\le n\}$ as its set of edges. The graph of $\Gamma_8$ is depicted below. This transition enabled Lazar and Wachs to apply the so called Rota-Whitney formula for evaluating the M\"obius function of the bond lattice $\Pi_{\Gamma_{2n}}$, and eventually to recover Hetyei's Theorem~\ref{thm:hetyei} and further results along this line.

\begin{figure*}
\centering
    \begin{tikzpicture}[scale=0.4]
    \draw[-] (0,7) to (4,6);
    \draw[-] (0,7) to (4,4);
    \draw[-] (0,7) to (4,2);
    \draw[-] (0,7) to (4,0);
    \draw[-] (0,5) to (4,4);
    \draw[-] (0,5) to (4,2);
    \draw[-] (0,5) to (4,0);
    \draw[-] (0,3) to (4,2);
    \draw[-] (0,3) to (4,0);
    \draw[-] (0,1) to (4,0);
    \node at (0,7){$\bullet$};
    \node at (0,5){$\bullet$};
    \node at (0,3){$\bullet$};
    \node at (0,1){$\bullet$};
    \node at (4,6){$\bullet$};
    \node at (4,4){$\bullet$};
    \node at (4,2){$\bullet$};
    \node at (4,0){$\bullet$};
    \node at (-.5,7){$1$};
    \node at (-.5,5){$3$};
    \node at (-.5,3){$5$};
    \node at (-.5,1){$7$};
    \node at (4.5,6){$2$};
    \node at (4.5,4){$4$};
    \node at (4.5,2){$6$};
    \node at (4.5,0){$8$};
    \end{tikzpicture} 
\end{figure*}

Here we take a different look at $\Gamma_{2n}$ and the associated arrangement $\cK_{2n}$. Interpreted appropriately, the parity restrictions in their definitions highly resemble the even-odd drop condition for permutations in $\cE_{2n}$ (see Definition~\ref{def:five perm}(6) below). This motivates us to look for a more direct connection between $\cK_{2n}$ and $\cE_{2n}$ and leads us to Theorem~\ref{thm:label-I}. Before we explain this further, let us shift our attention to various classes of restricted permutations that are related to Genocchi numbers and Genocchi medians. To give more context, we recall seven types of permutations. 

\begin{Def}\label{def:five perm}
For every $n\ge 1$, we introduce the following seven kinds of permutation subject to various parity restrictions.
    \begin{enumerate}
        \item A permutation $\sigma\in\S_{2n}$ is called a {\it Dumont permutation of the first kind} if it satisfies the condition that if $\sigma_i$ is even then $\sigma_i>\sigma_{i+1}$ and $i<2n$, otherwise $\sigma_i<\sigma_{i+1}$ or $i=2n$. The set of such permutations is denoted as $\D^{\mathrm{I}}_{2n}$.
        \item A permutation $\sigma\in\S_{2n}$ is called a {\it Dumont permutation of the second kind} if $\sigma_{2i}<2i$ and $\sigma_{2i-1}\ge 2i-1$ for all $i\in [n]$. The set of such permutations is denoted as $\D^{\mathrm{II}}_{2n}$.
        \item A permutation $\sigma\in\S_{2n}$ is called a {\it Dumont permutation of the third kind} if $\sigma_i>\sigma_{i+1}$ then $\sigma_i$ and $\sigma_{i+1}$ are both even. The set of such permutations is denoted as $\D^{\mathrm{III}}_{2n}$.
        \item A permutation $\sigma\in\S_{2n}$ is called a {\it Dumont permutation of the fourth kind} if $\sigma_i<i$ then $i$ and $\sigma_i$ are both even. The set of such permutations is denoted as $\D^{\mathrm{IV}}_{2n}$.
        \item A permutation $\sigma\in\S_{2n}$ is called a {\it $D$-permutation} if $\sigma_{2i}\le 2i$ and $\sigma_{2i-1}\ge 2i-1$ for all $i\in [n]$. The set of such permutations is denoted as $\fD_{2n}$.
        \item Denote the permutations of $[2n]$ with only even-odd drops by $\cE_{2n}$ and call them the {\it $E$-permutations}.
        \item Denote by $\cX_{2n}$ the set of permutations $\sigma\in\S_{2n}$ satisfying that if $\sigma_i>\sigma_{i+1}$ then $\sigma_i$ is even and $\sigma_{i+1}$ is odd.
    \end{enumerate}
\end{Def}
The enumeration of all seven classes of permutations above gives rise to the following relations.
\begin{align*}
&|\D^{\mathrm{I}}_{2n}|=|\D^{\mathrm{II}}_{2n}|=|\D^{\mathrm{III}}_{2n}|=|\D^{\mathrm{IV}}_{2n}|=g_{n+1},\\
&|\fD_{2n}|=|\cE_{2n}|=|\cX_{2n}|=h_n.
\end{align*}

Note that Dumont permutations of the first two kinds (items (1) and (2) above) were introduced by Dumont~\cite{dum74} and are alluded to in the introduction; items (3) and (4) were defined by Burstein, Josuat-Verg\`es, and Stromquist~\cite{BJS10}; items (5) and (6) were due to Lazar and Wachs~\cite{LW22}; item (7) was considered by both Eu-Fu-Lai-Lo~\cite{EFLL22} and Pan-Zeng~\cite{PZ23} and is seen to be in bijection with item (6) via the famed Foata's first fundamental transformation (see \cite{FS70,PZ23} and \cite[Prop.~2.11]{EFLL22}). Moreover, a permutation $\sigma=\sigma_1\cdots\sigma_{2n}$ has only odd-even ascents precisely when its reversal $\sigma^{\mathrm{r}}:=\sigma_{2n}\cdots\sigma_1$ has only even-odd descents, i.e., $\sigma\in \GI_{2n}$ if and only if $\sigma^{\mathrm{r}}\in\cX_{2n}$.

We end this section with a diagram that sets our Theorem~\ref{thm:label-I} in the context of existing results. It is evident from Fig.~\ref{relation diagram} that in view of Theorem~\ref{thm:label-I}, the two enumerative results due respectively to Hetyei (Theorem~\ref{thm:hetyei}) and Lazar-Wachs (Theorem~\ref{thm:Lazar-Wachs}) are essentially equivalent. Here $\Psi$ refers to a variant of Foata's first fundamental transformation~\cite[Chap.~10.2]{lot97} that we describe next. 

Take any permutation $\pi\in\cE_{2n}$, we write its inverse in cycle notation $\pi^{-1}=c_1c_2\cdots c_k$ such that each cycle $c_i$ starts with its minimal element and $\min(c_1)>\min(c_2)>\cdots>\min(c_k)$. Then we remove the parentheses of the cycles and let the new permutation (in its one-line notation) be $\Psi(\pi)$. For example, if $\pi=324165\in\cE_{6}$, then $\pi^{-1}=(56)(2)(143)$ and $\Psi(\pi)=562143\in\GI_6$. To recover the preimage under $\Psi$, simply note that the left-to-right minima of $\Psi(\pi)$ are precisely the minimal elements in each cycle of $\pi$. Furthermore, we note in passing that under this bijection $\Psi$, a permutation $\pi\in\cE_{2n}$ has only one cycle, if and only if its image $\Psi(\pi)$ begins with $1$. Combining this fact with \eqref{id:Seidel two Genocchi}, \eqref{id:1=2}, and the $k=1$ case of \eqref{id:GI and S}, we are rewarded with yet another proof of Lazar and Wachs's Conjecture~\ref{conj:Lazar-Wachs}. Similarly, combining the $k=n$ case of \eqref{id:GI and S} with \eqref{En2n=sum of above line}, Theorem~\ref{thm:label-I}, and Proposition~\ref{prop:K2n}, we can deduce Hetyei's Theorem~\ref{thm:hetyei}.

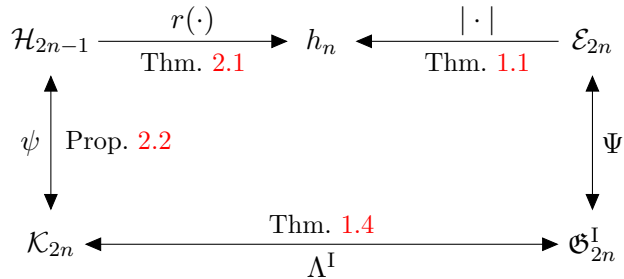
\begin{figure}
\begin{tikzpicture}[scale=0.9]
\draw(0,0) node{$\cK_{2n}$};
\draw(8,0) node{$\GI_{2n}$};
\draw(0,3) node{$\cH_{2n-1}$};
\draw(8,3) node{$\cE_{2n}$};
\draw(4,3) node{$h_n$};
\draw(8.3,1.5) node{$\Psi$};

\draw(2.1,3.3) node{$r(\cdot)$};
\draw(6.3,3.3) node{$|\cdot|$};
\draw(2.1,2.7) node{\small Thm.~\ref{thm:hetyei}};
\draw(6.3,2.7) node{\small Thm.~\ref{thm:Lazar-Wachs}};
\draw(4,0.3) node{\small Thm.~\ref{thm:label-I}};
\draw(4,-.3) node{$\LI$};
\draw(1,1.5) node{\small Prop.~\ref{prop:K2n}};
\draw(-.3,1.5) node{$\psi$};

\draw[>=triangle 45, <->] (0.5,0) -- (7.5,0);
\draw[>=triangle 45, ->] (0.7,3) -- (3.5,3);
\draw[>=triangle 45, ->] (7.5,3) -- (4.5,3);
\draw[>=triangle 45, <->] (0,0.5) -- (0,2.5);
\draw[>=triangle 45, <->] (8,0.5) -- (8,2.5);
\end{tikzpicture}
\caption{The equivalence between Theorems~\ref{thm:hetyei} and \ref{thm:Lazar-Wachs} \label{relation diagram}}
\end{figure}

\section{Parity pattern and the labelings with all four models}\label{sec:labeling}
\subsection{Four sets of parity pattern avoiding permutations}\label{sub:la.def}
Given a permutation $\pi=\pi_1\cdots\pi_n$, each pair $\pi_i>\pi_{i+1}$ (resp.~$\pi_i<\pi_{i+1}$) is called a {\it descent} (resp.~an {\it ascent}) of $\pi$. Let $\sigma$ be a permutation of length $k$. We say $\pi$ \textbf{contains} the pattern $\sigma$ if in $\pi$ there exists a subsequence of length $k$ that is order isomorphic to $\sigma$, otherwise we say $\pi$ \textbf{avoids} $\sigma$. From this perspective, descents and ascents are simply consecutive $21$ and $12$ patterns, respectively. For background and further information on the topic of pattern avoidance, we refer the reader to Kitaev's monograph~\cite{kit11} and the voluminous amount of references therein. 

We are going to restrict the parities of both numbers $\pi_i$ and $\pi_{i+1}$ that form either a descent or an ascent.  By a {\it parity pattern} we refer to one of the following eight consecutive patterns:
$$\eE,~\eO,~\oE,~\oO,~\Ee,~\Eo,~\Oe,\text{ and }\Oo.$$
Here the parity pattern, say $\eO$, refers to an ascent $\pi_i<\pi_{i+1}$ in $\pi$ such that $\pi_i$ is even and $\pi_{i+1}$ is odd. The remaining seven patterns are understood analogously. Notice that Kitaev and Remmel~\cite{KR07} has previously considered similar parity restrictions but they only restricted the parity of one of $\pi_i$ and $\pi_{i+1}$. Their work does involve Genocchi numbers though, which potentially gives rise to bijective questions. ({\SF Todo 2})

Recall from Definition~\ref{def:G-I} that a permutation is in $\GI_{2n}$ if and only if it avoids simultaneously the three parity patterns $\eE$, $\eO$, and $\oO$. Motivated by this definition as well as our first labeling $\LI$, we next introduce three more subsets of $\S_{2n}$ that are enumerated by the median Genocchi numbers. Their definitions not only vary in the choices of parity patterns to be avoided, but also require some additional restrictions. Note that however, $\eO$ is one common pattern to be avoided by all four subsets of permutations.


	

\begin{Def}\label{def:GII}
	For every $n\geq 1$, let $\GII_{2n}$ denote the set of permutations $\sigma$ in $\S_{2n}$ subject to the following conditions. 
	\begin{enumerate}
		\item $\sigma$ avoids the parity patterns $\Ee,\eO,\oO$;
		\item the only occurrence of $\Oe$ pattern that $\sigma$ allows is as follows: $\sigma_1>\cdots>\sigma_m>\sigma_{m+1}$, where $(\sigma_m,\sigma_{m+1})$ is an $\Oe$ pattern and all $\sigma_i~(1 \le i \le m)$ are odd numbers. We may call such $\Oe$ pattern an {\it initial $\Oe$ pattern}.
	\end{enumerate}  
\end{Def}

For example, the set $\GII_{4}$ includes the following eight permutations, among which $3214$ and $3241$ each has an initial $\Oe$ pattern. 
$$\GII_{4}=\{1243,2143,2431,3124,3214,3241,3412,4312\}.$$

To characterize our third kind of parity pattern avoiding permutations, we introduce the following generalized notion of $\eO$ pattern. 


\begin{Def}\label{def:gen eO}
	Given a permutation $\sigma$ and an integer $d\ge 0$, a pair $(\sigma_i,\sigma_j)$ is said to form an $\eO$ \emph{pattern at distance $d$} of $\sigma$, denoted as $\eO_d$, if the following conditions are satisfied.
	\begin{enumerate}
		\item $j=i+d+1$;
		\item $\sigma_i\equiv \sigma_l+1\equiv 0 \pmod 2$ for all $i<l\le j$;
		\item $\sigma_i<\sigma_j$.
	\end{enumerate}
\end{Def}
Note that the $\eO$ pattern at distance $0$ is the original (consecutive) $\eO$ pattern.

\begin{Def}
	For every $n\ge 1$, let $\GIII_{2n}$ denote the set of permutations $\sigma\in\S_{2n}$ such that the following conditions are satisfied.
	\begin{enumerate}
		\item $\sigma$ avoids the parity patterns $\Oo$ and $\eE$;
		\item $\sigma$ avoids the generalized parity pattern $\eO_d$ for any $d\ge 0$.
	\end{enumerate}
\end{Def}
For example, the set $\GIII_{4}$ consists of the following eight permutations, while the permutation $2134$ is not included since $(2,3)$ forms an $\eO_1$ pattern in it.
$$\GIII_{4}=\{1342,1432,2143,3214,3412,3421,4132,4321\}.$$

Our last model of permutations that are enumerated by the median Genocchi numbers needs the following extensions of the $\Oo$ and $\Ee$ patterns, in a similar vein as Definition~\ref{def:gen eO} extends the $\eO$ pattern. We begin with some notation concerning factors in a word. Given any word $w$, a \emph{factor} refers to a subword of $w$ that is formed by consecutive letters. In particular, a single letter is a factor of length $1$. A factor $\alpha$ is called \emph{even} (resp.~\emph{odd}) if it consists of only even (resps. odd) letters, and in that case we shall write $\alpha=\mathrm{E}$ (resp.~$\alpha=\mathrm{O}$). Let $\max(\alpha)$ (resp.~$\min(\alpha)$) denote the maximal (resp.~minimal) letter contained in $\alpha$.

\begin{Def}\label{def:gen Oo and Ee}
	Given a permutation $\sigma$ and an integer $d\ge 0$, a pair $(\sigma_i,\sigma_j)$ is said to form an $\Oo$ \emph{pattern at distance $d$} of $\sigma$, denoted as $\Oo_d$, if the following conditions are satisfied.
	\begin{enumerate}
		\item $\sigma=\cdots\sigma_i\alpha\sigma_j\cdots$;
		\item $\sigma_i\equiv\sigma_j \equiv 1\pmod 2$ and the factor $\alpha=\mathrm{E}$ is of length $d$;
		\item $\sigma_i>\sigma_j$ if $d=0$ and $\sigma_i>\max(\alpha)>\sigma_j$ if $d>0$.
	\end{enumerate}
	A pair $(\sigma_i,\sigma_j)$ is said to form an $\Ee$ \emph{pattern at distance $d$} of $\sigma$, denoted as $\Ee_d$, if the following conditions are satisfied.
	\begin{enumerate}
		\item $\sigma=\cdots\sigma_i\alpha\sigma_j\cdots$;
		\item $\sigma_i\equiv\sigma_j \equiv 0 \pmod 2$ and the factor $\alpha=\mathrm{O}$ is of length $d$;
		\item $\sigma_i>\sigma_j$ if $d=0$ and $\sigma_i>\min(\alpha)>\sigma_j$ if $d>0$.
	\end{enumerate}
\end{Def}
It is worth noting that $\Oo_0$ (resp. $\Ee_0$) is simply the original $\Oo$ (resp. $\Ee$) pattern.

\begin{Def}
	For every $n\ge 1$, let $\GIV_{2n}$ denote the set of permutations $\sigma\in\S_{2n}$ such that the following conditions are satisfied.
	\begin{enumerate}
		\item $\sigma$ avoids the parity pattern $\eO$;
		\item $\sigma$ avoids the generalized parity patterns $\Oo_d$ and $\Ee_d$ for any $d\ge 0$.
	\end{enumerate}
\end{Def}

For example, there are eight permutations in $\GIV_{4}$ as shown below, while the permutation $1432$ is not included since $(4,2)$ forms an $\Ee_1$ pattern in it.
$$\GIV_{4}=\{1243,1324,2134,2143,2413,3241,3412,4132\}.$$

All three newly introduced sets of permutations could be utilized to label the regions of $\cK_{2n}$, as indicated by the next theorem. From this perspective, they arise as natural companions to the first set $\GI_{2n}$, even though their definitions are admittedly more involved.

\begin{theorem}\label{thm:three more labelings}
For any $n\ge 1$, the regions of $\cK_{2n}$ are bijectively labeled, via the following three mappings, by the permutations in $\GII_{2n}$, $\GIII_{2n}$, and $\GIV_{2n}$, respectively.
\begin{align*}
\LII: &~\cR(\cK_{2n}) \to \GII_{2n},\\
\LIII: &~\cR(\cK_{2n}) \to \GIII_{2n},\\
\LIV: &~\cR(\cK_{2n}) \to \GIV_{2n}.
\end{align*}
\end{theorem}

\subsection{The four labelings}\label{sub:la.model-1}
 Given a point $\bx=(x_1,\ldots,x_{2n},x_{2n+1})$ that belongs to a region, say $R$, of $\cK_{2n}$, note that the coordinate $x_{2n+1}$ is unrestricted and can be ignored. For the remaining $2n$ coordinates, according to the region $R$ we know the sign of $x_i-x_j$ for each pair $(i,j)$ such that $i<j$ and $i+1\equiv j\equiv 0\pmod 2$. In what follows, we refer to such a pair as an ``$\oE$ pair''. 
 
 \begin{Def}\label{def:compatible}
 	A permutation $\pi$ in $\S_{2n}$ is said to be {\it compatible} with a given region $R\in\cR(\cK_{2n})$, if for every $\oE$ pair $(i,j)$ and $\bx\in R$, we have
 	$$\pi^{-1}(i)<\pi^{-1}(j) \text{ if and only if } x_i<x_j.$$ 
 \end{Def}
 
\begin{remark}
Each region $R\in\cR(\cK_{2n})$ uniquely determines a poset, say $P(R)$, on the elements $\{x_1,\ldots,x_{2n}\}$, and thus can be represented using the Hasse diagram associated with $P(R)$; see Fig.~\ref{Hasse diagram} for the Hasse diagram of a region in $\cR(\cK_{18})$. By Definition~\ref{def:compatible}, a permutation is compatible with $R$ if and only if it is a linear extension\footnote{To be precise, for each $1\le i\le 2n$, we identify the element $x_i$ in the poset with the letter $i$ in the permutation.} of the poset $P(R)$. From this perspective, each of the four sets $\GI_{2n}$, $\GII_{2n}$, $\GIII_{2n}$, and $\GIV_{2n}$ could be viewed as a collection of uniquely chosen representatives of the linear extensions of $P(R)$, when $R$ runs through all regions in $\cR(\cK_{2n})$.
\end{remark}

 
 \begin{proof}[Proof of Theorem~\ref{thm:label-I}]
 	We are going to construct recursively a bijection
 	$$\LI:=\LI_n: \cR(\cK_{2n}) \to \GI_{2n}.$$
 	Given a region $R$ of $\cK_{2n}$, we aim to show that there exists a unique permutation $\pi\in\GI_{2n}$ that is compatible with $R$, and let $\LI_n(R):=\pi$. The proof goes by induction on $n$. 
 	
 	For $n=1$, there is only one hyperplane $x_1-x_2=0$ in $\cK_2$ which divides $\bR^3$ into two regions, namely
 	$$R_1=\{\bx\in\bR^3: x_1-x_2<0\}, \text{ and } R_2=\{\bx\in\bR^3: x_1-x_2>0\}.$$
 	To satisfy the compatibility, we are forced to take $\LI_1(R_1)=12$ and $\LI_1(R_2)=21$. Now suppose $\LI_n$ has already been constructed for a certain $n\ge 1$, we proceed to define $\LI_{n+1}$ which completes the following commutative diagram.
 	
 	\begin{figure*}[h!]
 		\begin{tikzpicture}[scale=0.8]
 			\draw(0,0) node{$\cR(\cK_{2n+2})$};
 			\draw(5,0) node{$\GI_{2n+2}$};
 			\draw(0,3) node{$\cR(\cK_{2n})$};
 			\draw(5,3) node{$\GI_{2n}$};
 			\draw(5.7,1.5) node{IA-I};
 			
 			\draw(2.5,3.3) node{$\LI_n$};
 			\draw(2.5,0.3) node{$\LI_{n+1}$};
 			\draw(-.7,1.5) node{proj};

 			\draw[>=triangle 45, ->, dashed] (1,0) -- (4.3,0);
 			\draw[>=triangle 45, ->] (1,3) -- (4.3,3);
 			\draw[>=triangle 45, <-] (5,0.5) -- (5,2.5);
 			\draw[>=triangle 45, ->] (0,0.5) -- (0,2.5);
 		\end{tikzpicture}
 	\end{figure*}
 	
 	Fix a given region $R\in\cR(\cK_{2n+2})$, we first project\footnote{Note that going through this projection, we are deprived of the information on the magnitude comparisons between the coordinates $x_{2n+2}$ and $x_j$ for all $j=1,3,\ldots,2n+1$. This observation explains why we also need $\bx$ as an input for the insertion algorithm.} it to a region $\tilde{R}\in\cR(\cK_{2n})$ as follows: 
 	\begin{align*}
 		\proj:~R &\to \tilde{R} \\
 		\bx=(x_1,\cdots,x_{2n+1},x_{2n+2},x_{2n+3}) &\mapsto \tilde{\bx}=(x_1,\cdots,x_{2n+1}).
 	\end{align*}
 	Now by inductive hypothesis we can find the unique permutation $\tilde{\pi}:=\LI_{n}(\tilde{R})\in\GI_{2n}$ such that $\tilde{\pi}$ is compatible with $\tilde{R}$. Next, we insert $2n+1$ and $2n+2$ into $\tilde{\pi}$ while rearranging smaller letters when needed to derive a unique permutation $\pi\in\GI_{2n+2}$. Finally, we set $\LI_{n+1}(R):=\pi$. 
 	
 	Pick any vector $\bx\in R$, there are two cases to consider and the core of the construction is described as the {\it Insertion Algorithm I} and boxed below.
 	\begin{itemize}
 		\item If $x_{2n+1}<x_{2n+2}$ in $\bx$, then input $y=2n+2$ and run the insertion algorithm I to produce the word $w$. We define the permutation $\pi:=(2n+1)w$.
 		\item If $x_{2n+1}>x_{2n+2}$ in $\bx$, then input the concatenated $y=(2n+2)(2n+1)$ and run the insertion algorithm I to produce the word $w$. We define the permutation $\pi:=w$.
 	\end{itemize}
 	
 	A pair of odd integers $(i,j)$ is said to be {\it bad} with respect to $\bx$, if $i$ precedes $j$ in the one-line expression of $\tilde{\pi}$ but $x_j<x_{2n+2}<x_i$ in $\bx$. The following facts about bad pairs will be useful for later discussion.

 	\begin{claim}\label{cla:bad pair}
 	Suppose the pair $(i,j)$ is bad in $\tilde{\pi}$ with respect to $\bx$, then we claim that
 	\begin{enumerate}
 		\item $i>j$;
 		\item for every even number $k$ that appears between $i$ and $j$, we have $i>k$.
 	\end{enumerate}
 	\end{claim}

 	To see item (1), assume on the contrary that $i<j$, recall that $\tilde{\pi}$ avoids parity patterns $\oO$ and $\eO$, hence there exists at least one even entry, say $k$, between $i$ and $j$ such that $k>j$. But this implies that $x_i<x_k<x_j$ in both $\bx$ and $\tilde{\bx}$, which contradicts with $x_j<x_{2n+2}<x_i$. For item (2), if $i<k$ then combining with item (1) we see $k$ is an even number larger than both $i$ and $j$, hence we arrive at the same contradictory inequality $x_i<x_k<x_j$.

 	The algorithm below tells us how to remedy all the bad pairs by rearranging certain odd entries in $\tilde{\pi}$, and then insert $2n+1,2n+2$ into the adjusted permutation. 
 	\begin{framed}
 		\begin{center}
 			{\bf Insertion Algorithm I (IA-I)}
 		\end{center}
 		\begin{itemize}
 			\item[Initiate] Input $y$, $\tilde{\pi}$, and $\bx$.
 			\item[Step 1] If there are no bad pairs in $\tilde{\pi}$ with respect to $\bx$, then let $\tilde{w}:=\tilde{\pi}$ and jump to Step 4. Otherwise, find the rightmost $j$ in $\tilde{\pi}$ such that $(i,j)$ forms a bad pair with respect to $\bx$ for a certain $i$.
 			\item[Step 2] Find all bad pairs ending with $j$, say $(i_1,j),\ldots,(i_m,j)$ with $i_1>\cdots>i_m$. Remove $i_1,\ldots,i_m$ from $\tilde{\pi}$ and form the subword $u:=i_1 i_2 \cdots i_m$ by concatenation.
 			\item[Step 3] Insert $u$ to the immediate right of $j$ to get $\tilde{w}:=\cdots ju\cdots$.
 			\item[Step 4] Find the rightmost odd number $k$ in $\tilde{w}$ such that $x_k<x_{2n+2}$. Insert $y$ to the immediate right of $k$ to obtain $w:=\cdots ky\cdots$. Note that if the algorithm goes through Steps 2 and 3, then $k=j$. In the particular case that no such $k$ exists, simply set $w:=y\tilde{w}$.
 			\item[End] Output $w$.
 		\end{itemize}
 	\end{framed}
 	To finish the proof, one needs to verify the following facts.
 	\begin{itemize}
 		\item $\pi\in\GI_{2n+2}$. It suffices to show that the newly introduced ascents after applying IA-I are all $\oE$ ascents. There are three scenarios that potential new ascents could arise from.
 		\begin{enumerate}
 			\item When $i_1,\cdots,i_m$ are removed in step 2. Suppose that for a certain $1\le t\le m$, $i_t$ is removed from $\cdots ai_tb\cdots$ and $ab$ does form a new ascent. Recall that $i_t$ is odd and $\tilde{\pi}$ avoids both $\oO$ and $\eO$, hence $i_t<a<b$ and $b$ must be even. But now $b$ is an even sitting inbetween $i_t$ and $j$, according to Claim~\ref{cla:bad pair} (2) we should have $i_t>b$. This is a contradiction.
 			
 			\item When $u=i_1\cdots i_m$ is inserted in step 3, say as $\cdots j i_1\cdots i_m l\cdots$, then $ji_1$ and $i_ml$ may form new ascents. For $ji_1$, it will be separated by the $y$ inserted in step 4\footnote{Notice that if steps 2 and 3 were executed, then the odd number $k$ in step 4 is precisely the number $j$ found in step 2.}. While for $i_ml$, by Claim~\ref{cla:bad pair} (1) we have $i_m>j$, so if $i_ml$ does form an ascent then $jl$ is already an ascent in $\tilde{\pi}$, which forces the ascent top $l$ to be even. We see indeed $i_ml$ is an $\oE$ ascent that is allowed.

 			\item When $y$ is inserted after $k$ in step 4, a new ascent may be formed between $k$ and $y$. But no matter $y=2n+2$ or $y=(2n+2)(2n+1)$, it must begin with the even number $2n+2$, so this new ascent $k(2n+2)$ is a legal $\oE$ ascent.
 		\end{enumerate}

 		\item $\pi$ is compatible with the region $R$. The algorithm clearly ensures that for any $1\le i\le n+1$, $\pi^{-1}(2i-1)<\pi^{-1}(2n+2)$ if and only if $x_{2i-1}<x_{2n+2}$ in $\bx$. It remains to verify such compatibility for the pair $(2i-1,2j)$, $1\le i\le j\le n$. This holds true for $\tilde{\pi}$ by inductive hypothesis. And thanks to Claim~\ref{cla:bad pair} (2), we know that although certain $2i-1$ might have been moved by the algorithm, the relative orders between $2i-1$ and $2j$ (i.e., which precedes which) remain the same in both $\tilde{\pi}$ and $\pi$.

 		\item $\LI_{n+1}$ is injective. Suppose $R$ and $R'$ are two distinct regions in $\cR(\cK_{2n+2})$, then there exists at least one hyperplane, say $x_{2i-1}-x_{2j}=0$ for certain $1\le i\le j\le n+1$, which separates $R$ and $R'$. Applying compatibility, we see that the relative orders between $2i-1$ and $2j$ are opposite in $\LI_{n+1}(R)$ and $\LI_{n+1}(R')$, therefore $\LI_{n+1}(R)\neq \LI_{n+1}(R')$.
 		
 		\item The permutation in $\GI_{2n+2}$ that is compatible with the given region $R$ is unique. Suppose $\pi,\pi'\in\GI_{2n+2}$ are two distinct permutations both being compatible with $R$. Find a pair $(a,b)$ with $a<b$ such that $a$ precedes $b$ in $\pi$ while $b$ precedes $a$ in $\pi'$. There are four cases to be discussed, conditioning on the parities of $a$ and $b$.
 		\begin{itemize}
 			\item[(i)] $a\equiv b\equiv 0\pmod 2$. Since $\pi$ avoids $\eE$, there must be some odd numbers lying between $a$ and $b$ in $\pi$. Suppose $\pi=\cdots aa_1\cdots a_t c\cdots b\cdots$, where $a_1\equiv\cdots\equiv a_t\equiv c+1\equiv 0\pmod 2$, then $a>a_1>\cdots>a_t>c$ due to the avoidance of $\eE$ and $\eO$, thus in region $R$ we have $x_a<x_c<x_b$, rendering $\pi'$ incompatible with $R$, a contradiction.
 			\item[(ii)] $a+1\equiv b+1\equiv 0\pmod 2$. Since $\pi$ avoids $\oO$, there must be some even numbers lying between $a$ and $b$ in $\pi$. Suppose $\pi=\cdots a\cdots c b_1\cdots b_t b\cdots$, where $b_1+1\equiv\cdots\equiv b_t+1\equiv c\equiv 0\pmod 2$, then $c>b_1>\cdots>b_t>b$ due to the avoidance of $\oO$ and $\eO$, thus in region $R$ we have $x_a<x_c<x_b$, rendering $\pi'$ incompatible with $R$, a contradiction.
 			\item[(iii)] $a\equiv b+1\equiv 0\pmod 2$. Analogous to case (i), we find the nearest odd number to the right of $a$ and suppose $\pi=\cdots aa_1\cdots a_t c\cdots b\cdots$, where $a_1\equiv\cdots\equiv a_t\equiv c+1\equiv 0\pmod 2$, again we have $a>a_1>\cdots>a_t>c$ due to the avoidance of $\eE$ and $\eO$. Applying compatibility, we see $x_a<x_c$ in region $R$ and $a$ should precede $c$ in $\pi'$ as well. This in turn implies that $b$ precedes $c$ in $\pi'$. Now we could view $(c,b)$ as a pair in case (ii), giving rise to a similar contradiction.
 			\item[(iv)] $a+1\equiv b\equiv 0\pmod 2$. This case cannot happen, since the relative order between $a$ and $b$ is predetermined by the side of the hyperplane $x_a-x_b=0$ that $R$ appears on, so this order should be consistent in both $\pi$ and $\pi'$.
 		\end{itemize}

 		\item The uniqueness from the last item also implies that $\LI_{n+1}$ is surjective. Indeed, given a permutation $\sigma\in\GI_{2n+2}$, utilizing the relative orders between each pair $(2i-1,2j)$, $1\le i\le j\le n+1$, one finds a unique region $R\in \cR(\cK_{2n+2})$ that is compatible with $\sigma$. Now we have shown that $\LI_{n+1}(R)\in\GI_{2n+2}$ and $\LI_{n+1}(R)$ is compatible with $R$. Relying on the uniqueness from the last item, we deduce that $\LI_{n+1}(R)=\sigma$.

 	\end{itemize}
 \end{proof}
 
 \begin{example}\label{ex-I}
 	To illustrate the first labeling $\LI$ constructed in the proof of Theorem~\ref{thm:label-I}, we start with a specific region $R\in\cR(\cK_{18})$ given below, and go through IA-I step-by-step to find its image $\pi$ in $\GI_{18}$.
 	\begin{align*}
 	R=\{ & \bx\in\bR^{19}: x_{2i-1}-x_{2j_1}<0, x_{2i-1}-x_{2j_2}>0,\\
 	& j_1=1,2,4,6,7,8,9,~j_2=3,5 \text{ when } i=1,\\
 	& j_1=2,6,7,8,~j_2=3,4,5,9 \text{ when } i=2,\\
 	& j_1=4,5,6,7,8,9,~j_2=3 \text{ when } i=3,\\
 	& j_1=4,6,7,8,~j_2=5,9 \text{ when } i=4,\\
 	& j_1=8,~j_2=5,6,7,9 \text{ when } i=5,\\
 	& j_1=6,8,~j_2=7,9\text{ when } i=6,\\
 	& j_1\in \varnothing,~j_2=7,8,9\text{ when }i=7,\\
 	& j_1=8,~j_2=9 \text{ when } i=8,\\
 	& j_1\in \varnothing,~j_2=9 \text{ when } i=9.\}
 	\end{align*}
 	This region can be succinctly expressed using the Hasse diagram in Fig.~\ref{Hasse diagram}, where $x_a$ connected with $x_b$ and lying above $x_b$ means that $x_a>x_b$.

 	To walk the reader through the construction of IA-I, we will show the whole process starting from $n=1$. At the beginning, we will give the Hasse diagram of $R$, as shown in Fig.~\ref{Hasse diagram}, which makes it easy to find the information on the magnitude comparisons between the coordinates.
 	\begin{figure*}[h!]
 		\centering
 		\begin{tikzpicture}[scale=0.4]
 			\draw[-] (0,0) to (0,19.5);
 			\draw[-] (0,4.5) to (4,7);
 			\draw[-] (0,6) to (4,8.5);
 			\draw[-] (0,10.5) to (4,13);
 			\draw[-] (0,6) to (-4,14);
 			\draw[-] (-4,14) to (0,18);
 			\node at (0,0){$\bullet$};
 			\node at (0,1.5){$\bullet$};
 			\node at (0,3){$\bullet$};
 			\node at (0,4.5){$\bullet$};
 			\node at (0,6){$\bullet$};
 			\node at (0,7.5){$\bullet$};
 			\node at (0,9){$\bullet$};
 			\node at (0,10.5){$\bullet$};
 			\node at (0,12){$\bullet$};
 			\node at (0,13.5){$\bullet$};
 			\node at (0,15){$\bullet$};
 			\node at (0,16.5){$\bullet$};
 			\node at (0,18){$\bullet$};
 			\node at (0,19.5){$\bullet$};
 			\node at (4,7){$\bullet$};
 			\node at (4,8.5){$\bullet$};
 			\node at (4,13){$\bullet$};
 			\node at (-4,14){$\bullet$};
 			\node at (1,0){$x_6$};
 			\node at (1,1.5){$x_5$};
 			\node at (1,3){$x_{10}$};
 			\node at (-1,4.5){$x_1$};
 			\node at (-1,6){$x_{18}$};
 			\node at (1,7.5){$x_7$};
 			\node at (1,9){$x_8$};
 			\node at (-1,10.5){$x_3$};
 			\node at (1,12){$x_{14}$};
 			\node at (1,13.5){$x_{11}$};
 			\node at (1,15){$x_{12}$};
 			\node at (1,16.5){$x_9$};
 			\node at (1,18){$x_{16}$};
 			\node at (1,19.5){$x_{13}$};
 			\node at (-5,14){$x_{15}$};
 			\node at (5,13){$x_4$};
 			\node at (5,8.5){$x_{17}$};
 			\node at (5,7){$x_2$};
 		\end{tikzpicture}
 		\caption{The Hasse diagram of a region $R\in\cR(\cK_{18})$ \label{Hasse diagram}}
 	\end{figure*}
 	
 	Define $\tilde{R}_i=\proj^{8-i}(R)$ for $i=0,\ldots,7$ and $\tilde{\pi}^{(j)}=\LI_{j}(\tilde{R}_{j-1})$ for $j=1, \ldots, 8$. One reads from the diagram above that $\tilde{R}_0=\{\bx\in\bR^3:x_1-x_2<0\}$, so initially $\tilde{\pi}^{(1)}=\LI_1(\tilde{R}_0)=12$, and we will insert $3$ and $4$. We see $x_3<x_4$ in the diagram so we set $y=4$ and start IA-I. There are no bad pairs in $\tilde{\pi}^{(1)}$ with respect to $\bx$ and $1$ is the rightmost $k$ in $\tilde{\pi}^{(1)}$ such that $x_k<x_4$, so 
 	$$k=1,~w=142,\text{ and }\tilde{\pi}^{(2)}=3142.$$
 	Next, we insert $5$ and $6$. The diagram shows $x_5>x_6$, so we set $y=65$, there are no bad pairs and there exists no $k$ such that $x_{k}<x_6$, so
 	$$\tilde{\pi}^{(3)}=w=653142.$$
 	We continue to insert $7$ and $8$. The diagram shows $x_7<x_8$, so $y=8$. Now $(3,1)$ is a bad pair since $3$ precedes $1$ in $\tilde{\pi}^{(3)}$ while the diagram indicates that $x_1<x_8<x_3$. Consequently, applying IA-I we derive
 	$$j=k=1,~u=3,~\tilde{w}=651342,~w=6518342,\text{ and }\tilde{\pi}^{(4)}=76518342.$$
 	The remaining numbers can be inserted analogously using IA-I so we omit the details. The key parameters for each round of IA-I are listed below to reflect the incremental changes. 
 	\begin{align*}
 	& y=(10)9,~j=k=5,~u=7,~\tilde{w}=65718342, \text{ and }\tilde{\pi}^{(5)}=w=65(10)9718342.\\
 	& y=(12),~j=k=3,~u=9,~\tilde{w}=65(10)7183942,\\
 	&\qquad \qquad \qquad w=65(10)7183(12)942,\text{ and }\tilde{\pi}^{(6)}=(11)65(10)7183(12)942.\\
	& y=(14)(13),~j=k=3,~u=(11),~\tilde{w}=65(10)7183(11)(12)942, \text{ and}\\
	& \qquad \qquad \qquad \tilde{\pi}^{(7)}=w=65(10)7183(14)(13)(11)(12)942.\\
	& y=(16),~j=k=9,~u=(13),~\tilde{w}=65(10)7183(14)(11)(12)9(13)42,\\
	& \qquad w=65(10)7183(14)(11)(12)9(16)(13)42,\text{ and }\tilde{\pi}^{(8)}=(15)65(10)7183(14)(11)(12)9(16)(13)42.\\
	& y=(18)(17),~j=k=1,~u=(15)7,~\tilde{w}=65(10)1(15)783(14)(11)(12)9(16)(13)42, \text{ and }\\
	& \qquad \pi=w=65(10)1(18)(17)(15)783(14)(11)(12)9(16)(13)42.
 	\end{align*}
 	
 \end{example}
 	
Next we move on to the proof of Theorem~\ref{thm:three more labelings}. To that end, we need to construct three more labelings $\LII$, $\LIII$, and $\LIV$. The process largely resembles that of $\LI$. 

For ease of writing, we introduce the following notation.

\begin{Def}
	 Given a word $\alpha=a_1a_2\cdots a_n$ composed of positive integers, we define $\mathrm{Re}(\alpha)$ as the new word obtained by reordering the letters $a_1,\cdots,a_n$ from the smallest to the largest. 
\end{Def}
 
 \begin{proof}[Proof of Theorem~\ref{thm:three more labelings}]
 	Since the construction of these three labelings is similar to that of $\LI$, we only sketch the ideas and emphasize the differences between these insertion algorithms. 
 	
 	For $\LII:=\LII_n: \cR(\cK_{2n}) \to \GII_{2n}$, it suffices to define $\LII_{n+1}$ which completes the following diagram. 

 	\begin{figure*}[h!]
 		\begin{tikzpicture}[scale=0.8]
 			\draw(0,0) node{$\cR(\cK_{2n+2})$};
 			\draw(5,0) node{$\GII_{2n+2}$};
 			\draw(0,3) node{$\cR(\cK_{2n})$};
 			\draw(5,3) node{$\GII_{2n}$};
 			\draw(5.7,1.5) node{IA-II};
 			
 			\draw(2.5,3.3) node{$\LII_n$};
 			\draw(2.5,0.3) node{$\LII_{n+1}$};
 			\draw(-.7,1.5) node{proj};

 			\draw[>=triangle 45, ->, dashed] (1,0) -- (4.3,0);
 			\draw[>=triangle 45, ->] (1,3) -- (4.3,3);
 			\draw[>=triangle 45, <-] (5,0.5) -- (5,2.5);
 			\draw[>=triangle 45, ->] (0,0.5) -- (0,2.5);
 		\end{tikzpicture}
 	\end{figure*}
 	
 	Similarly, by inductive hypothesis we can find the unique permutation $\tilde{\pi}:=\LII_{n}(\tilde{R})\in\GII_{2n}$. Next, we insert $2n+1$ and $2n+2$ into $\tilde{\pi}$ with modifications on smaller letters when needed to derive a unique permutation $\pi\in\GII_{2n+2}$. Finally, we set $\LII_{n+1}(R):=\pi$. 
 	
 	Note that Claim~\ref{cla:bad pair} for bad pairs still holds since $\tilde{\pi}$ avoids parity patterns $\oO$ and $\eO$.
 	
 	Pick any vector $\bx\in R$, there are also two cases to consider and the {\it insertion algorithm II} is boxed below as well.
    \begin{itemize}
    	\item If $x_{2n+1}<x_{2n+2}$ in $\bx$, then input $y=2n+2$ and run the insertion algorithm II to produce the word $w$. Let the permutation $\pi:=(2n+1)w$.
    	\item If $x_{2n+1}>x_{2n+2}$ in $\bx$, then input the joint letters $y=(2n+2)(2n+1)$ and run the insertion algorithm II to produce the word $w$. Let the permutation $\pi:=w$.
    \end{itemize}
    \begin{framed}
    	\begin{center}
    		{\bf Insertion Algorithm II (IA-II)}
    	\end{center}
    	\begin{itemize}
    		\item[Initiate] Input $y$, $\tilde{\pi}$, and $\bx$.
    		\item[Step 1] If there are no bad pairs in $\tilde{\pi}$ with respect to $\bx$, then let $\tilde{w}:=\tilde{\pi}$ and jump to Step 4. Otherwise, find the rightmost $j$ in $\tilde{\pi}$ such that $(i,j)$ forms a bad pair with respect to $\bx$ for a certain $i$.
    		\item[Step 2] Find all bad pairs ending with $j$, say $(i_1,j),\ldots,(i_m,j)$ with $i_1>\cdots>i_m$. Remove $i_1,\ldots,i_m$ from $\tilde{\pi}$ and form the subword $u:=i_1 i_2 \cdots i_m$ by concatenation.
    		\item[Step 3] Scan all the letters to the right of $j$, find the leftmost $z$ such that either 1) $z$ is odd; or 2) $z$ is even and greater than $i_m$. Insert $u$ to the immediate left of $z$ to get $\tilde{w}:=\cdots uz\cdots$. In the particular case that no such $z$ exists, simply append $u$ to the right end to get $\tilde{w}$. 
    		\item[Step 4] Find the leftmost odd number $k$ in $\tilde{w}$ such that $x_k>x_{2n+2}$. Insert $y$ into $\tilde{w}$ to the immediate left of $k$ to obtain $w:=\cdots yk\cdots$. Note that if the algorithm goes through Steps 2 and 3, then $k=i_1$. In the particular case that no such $k$ exists, simply take $w:=\tilde{w}y$.
    		\item[End] Output $w$.
    	\end{itemize}
    \end{framed}
    
    To finish the proof for $\LII$, one needs to verify the following facts, while the other details are similar to those in the proof of Theorem~\ref{thm:label-I} and thus are omitted.
    \begin{itemize}
    	\item $\pi\in\GII_{2n+2}$. It suffices to show that no steps in the algorithm will incur an ascent or a descent that is not allowed in $\GII$. There are three scenarios that potential new ascents or descents could arise from. 
    	\begin{enumerate}
    		\item When $i_1,\cdots,i_m$ are removed in step 2. Suppose that for a certain $1\le t\le m$, $i_t$ is removed from $\cdots ai_tb\cdots$ and $ab$ forms a new consecutive pair. Recall that $i_t$ is odd and $\tilde{\pi}$ avoids $\eO$ and $\oO$, hence $a>i_t$ and there are two subcases:
    		\begin{itemize}
    			\item $b$ is even. Then according to Claim~\ref{cla:bad pair} (2) we should have $a>i_t>b$. But now $i_tb$ becomes an $\Oe$ pattern in $\tilde{\pi}$ and it is not initial (in view of the even letter $a$ preceding $i_t$). This leads to a contradiction according to Definition~\ref{def:GII} (2).
    			\item $b$ is odd. Then $i_t>b$ since $\tilde{\pi}$ avoids $\oO$. Consequently, $ab$ forms either an $\Oo$ pattern or an $\Eo$ pattern, both are allowed.
    		\end{itemize}
    		
    		\item When $u=i_1\cdots i_m$ is inserted in step 3, say as 
    		$$\cdots j~\overbracket{\cdots l}^{v}~\overbracket{i_1\cdots i_m}^{u}~z\cdots,$$ 
    		then $li_1$ and $i_mz$ form new consecutive pairs. For $li_1$, it will be separated by the $y$ inserted in step 4. While for $i_mz$, if $z>i_m$ is even, then $i_mz$ forms an $\oE$ pattern which is allowed; otherwise $z$ must be odd, then by our choice of $z$ in step 3, we see that either $v=\varnothing$ and $i_m>j>z$ since $\tilde{\pi}$ avoids $\oO$, or $v$ consists of even letters smaller than $i_m$, and $\tilde{\pi}$ avoiding $\eO$ implies that $i_m>l>z$. In both cases we see indeed $i_mz$ forms an $\Oo$ pattern which is allowed.
    		
    		\item When $y$ is inserted before $k$ or after $\tilde{w}$ in step 4, the permutation is still legal since the immediate left of $y$ is always smaller than $y$ and $y$ begins with the even number $2n+2$, while the immediate right of $y$ is an smaller odd or non-existent. 
    	\end{enumerate}
    	
    	\item The permutation in $\GII_{2n+2}$ that is compatible with a given region $R$ is unique. Suppose $\pi,\pi'\in\GII_{2n+2}$ are two distinct permutations both compatible with $R$. Find a pair $(a,b)$ with $a<b$ such that $a$ precedes $b$ in $\pi$ while $b$ precedes $a$ in $\pi'$. We elaborate on only two cases, since the other two cases can be discussed in a similar fashion as in the proof of Theorem~\ref{thm:label-I} cases (ii) and (iv).
    	\begin{itemize}
    		\item[(i)] $a\equiv b\equiv 0\pmod 2$. Since $\pi'$ avoids $\Ee$, there must be some odd numbers lying between $b$ and $a$ in $\pi'$. Suppose $\pi'=\cdots b\cdots ca_1\cdots a_ta\cdots$, where $a_1\equiv\cdots\equiv a_t\equiv c+1\equiv 0\pmod 2$, then $a>a_t>\cdots>a_1>c$ due to the avoidance of $\Ee$ and non-initial $\Oe$, thus in region $R$ we have $x_b<x_c<x_a$, rendering $\pi$ incompatible with $R$, a contradiction.
    		\item[(ii)] $a\equiv b+1\equiv 0\pmod 2$. Since $\pi$ avoids $\eO$, we find the nearest even number to the left of $b$ and suppose $\pi=\cdots a\cdots cb_1\cdots b_tb\cdots$, where $b_1+1\equiv\cdots\equiv b_t+1\equiv c\equiv 0\pmod 2$, again we have $c>b_1>\cdots>b_t>b$ due to the avoidance of $\oO$ and $\eO$. Applying compatibility, we see $x_c<x_b$ in region $R$ and $c$ should precede $b$ in $\pi'$ as well. This in turn implies that $c$ precedes $a$ in $\pi'$. Now we could view $(a,c)$ as a pair in case (i), giving rise to a similar contradiction.
    	\end{itemize}
    \end{itemize} 
    
    For $\LIII:=\LIII_n: \cR(\cK_{2n}) \to \GIII_{2n}$, it suffices to define $\LIII_{n+1}$ which completes the following diagram. 
    
    \begin{figure*}[h!]
    	\begin{tikzpicture}[scale=0.8]
    		\draw(0,0) node{$\cR(\cK_{2n+2})$};
    		\draw(5,0) node{$\GIII_{2n+2}$};
    		\draw(0,3) node{$\cR(\cK_{2n})$};
    		\draw(5,3) node{$\GIII_{2n}$};
    		\draw(5.7,1.5) node{IA-III};
    		
    		\draw(2.5,3.3) node{$\LIII_n$};
    		\draw(2.5,0.3) node{$\LIII_{n+1}$};
    		\draw(-.7,1.5) node{proj};

    		\draw[>=triangle 45, ->, dashed] (1,0) -- (4.3,0);
    		\draw[>=triangle 45, ->] (1,3) -- (4.3,3);
    		\draw[>=triangle 45, <-] (5,0.5) -- (5,2.5);
    		\draw[>=triangle 45, ->] (0,0.5) -- (0,2.5);
    	\end{tikzpicture}
    \end{figure*}
    
    Similarly, by inductive hypothesis we can find the unique permutation $\tilde{\pi}:=\LIII_{n}(\tilde{R})\in\GIII_{2n}$. Next, we insert $2n+1$ and $2n+2$ into $\tilde{\pi}$ with necessary modifications when needed to derive a unique permutation $\pi\in\GIII_{2n+2}$. Finally, we set $\LIII_{n+1}(R):=\pi$. 
    
    Pick any vector $\bx\in R$, we input $y=2n+2$ and run the insertion algorithm III as boxed below. To insert $2n+1$ and derive $\pi$, there are two cases to consider. 
    \begin{itemize}
    	\item If $x_{2n+1}<x_{2n+2}$ in $\bx$, then we obtain the permutation $\pi$ from $w$ by inserting $2n+1$ to the immediate left of the leftmost even number in $w$.
    	\item If $x_{2n+1}>x_{2n+2}$ in $\bx$, then we obtain the permutation $\pi$ from $w$ by scanning the numbers to the right of $2n+2$, and inserting $2n+1$ to the immediate left of the leftmost even number we found.
    \end{itemize}
    
    But the facts about bad pairs are slightly different from what they were before, so we make a new claim to accomodate the current situation.
    
    \begin{claim}\label{cla:bad pair III}
    	Suppose the pair $(i,j)$ is bad in $\tilde{\pi}$ with respect to $\bx$, and there is at least one even number between $i$ and $j$. Then we claim that
    	\begin{enumerate}
    		\item $i>j$;
    		\item for every even number $k$ that appears between $i$ and $j$, we have $i>k$.
    	\end{enumerate}
    \end{claim}
    
    To see item (1), assume on the contrary that $i<j$, let $c$ be the nearest even number to the left of $j$, recall that $\tilde{\pi}$ avoids parity pattern $\eO_d$, hence $c>j$. But this implies that $x_i<x_c<x_j$ in both $\bx$ and $\tilde{\bx}$, which contradicts with $x_j<x_{2n+2}<x_i$. For item (2), if $i<k$ then combining with item (1) we see $k$ is an even number larger than both $i$ and $j$, hence we arrive at the same contradictory inequality $x_i<x_k<x_j$. 
    
    \begin{framed}
    	\begin{center}
    		{\bf Insertion Algorithm III (IA-III)}
    	\end{center}
    	\begin{itemize}
    		\item[Initiate] Input $y$, $\tilde{\pi}$, and $\bx$.
    		\item[Step 1] If there are no bad pairs in $\tilde{\pi}$ with respect to $\bx$, then let $\tilde{w}:=\tilde{\pi}$ and jump to Step 4. Otherwise, find the rightmost $j$ in $\tilde{\pi}$ such that $(i,j)$ forms a bad pair with respect to $\bx$ for a certain $i$.
    		\item[Step 2] Find all bad pairs ending with $j$, say $(i_1,j),\ldots,(i_m,j)$ with $i_1<\cdots<i_m$. Remove $i_1,\ldots,i_m$ from $\tilde{\pi}$ and form the subword $u:=i_1 i_2 \cdots i_m$ by concatenation.
    		\item[Step 3] Suppose the word obtained from Step 2 is $\cdots jvz\cdots$, where $v$ is an odd factor and $z$ is even. If no such $z$ exists, simply let $v$ be the entire suffix after $j$. Set $\tilde{w}:=\cdots j\mathrm{Re}(vu)z\cdots$ (or $\tilde{w}:=\cdots j\mathrm{Re}(vu)$ for the ``no such $z$ exists'' case).
    		\item[Step 4] Find the rightmost odd number $k$ in $\tilde{w}$ such that $x_k<x_{2n+2}$. Insert $y$ into $\tilde{w}$ to the immediate right of $k$ to obtain $w:=\cdots ky\cdots$. Note that if the algorithm goes through Steps 2 and 3, then $k=j$. In the particular case that no such $k$ exists, simply take $w:=y\tilde{w}$.
    		\item[End] Output $w$.
    	\end{itemize}
    \end{framed}
    
    To finish the proof for $\LIII$, one needs to verify the following facts. 
    \begin{itemize}
    	\item $\pi\in\GIII_{2n+2}$. There are four scenarios that potential new patterns could arise from. 
    	\begin{enumerate}
    		\item When $i_1,\cdots,i_m$ are removed in step 2. Suppose that for a certain $1\le t\le m$, $i_t$ is removed from $\cdots ai_tb\cdots$ and $ab$ forms a new consecutive pair. We shall consider the following three subcases.
    		\begin{itemize}
    			\item $a\equiv b+1\equiv 0$. To avoid the $\eO_d$ pattern we must have $a>b$, so $ab$ forms an $\Eo$ pattern which is allowed.
    			\item $a\equiv b\equiv 0$. Considering $\eO_d$ avoidance and applying Claim~\ref{cla:bad pair III} (2), we have $a>i_t>b$. So $ab$ forms an $\Ee$ pattern which is allowed.
    			\item $a$ is odd. In this case, all three possible patterns for $ab$, namely $\oO$, $\oE$, or $\Oe$ are allowed.
    		\end{itemize}
    		
    		\item When $u$ is inserted in step 3. Firstly the reordering $\Re(vu):=r_1\cdots r_s$ ensures that this subword avoids $\Oo$. Moreover, the consecutive pair $jr_1$ will be further separated by $y$ in step 4, while the pair $r_sz$ is allowed since $z$ is even and $r_s$ is odd.
    		
    		\item When $y$ is inserted in step 4, say as $\cdots kyl\cdots$. The pair $ky$ is allowed since $k$ is odd and smaller than $y$. The pair $yl$ is allowed since $y$ is even and larger than $l$. Moreover, no potential $\eO_d$ pattern could arise from this insertion since $y=2n+2$ itself is even and larger than all remaining numbers. 

    		\item When $2n+1$ is inserted into $w$ after running IA-III. For the case of $x_{2n+1}<x_{2n+2}$, $2n+1$ is followed by the leftmost even number, so no forbidden patterns could arise. For the case of $x_{2n+1}>x_{2n+2}$, no potential $\eO_d$ pattern could arise since the nearest even number to the left of $2n+1$ is $2n+2$, larger than $2n+1$.
    	\end{enumerate}
    	
    	\item The permutation in $\GIII_{2n+2}$ that is compatible with the given region $R$ is unique. Suppose $\pi,\pi'\in\GIII_{2n+2}$ are two distinct permutations both compatible with $R$. Find a pair $(a,b)$ with $a<b$ such that $a$ precedes $b$ in $\pi$ while $b$ precedes $a$ in $\pi'$. We again give details for only two cases, since the other two cases can be discussed in the same way as in the proof of Theorem~\ref{thm:label-I} cases (i) and (iv).
    	\begin{itemize}
    		\item[(i)] $a\equiv b\equiv 0\pmod 2$. Relying on the avoidance of $\eE$ and $\eO_d$, the proof parallels that of case (i) in the proof of Theorem~\ref{thm:label-I}.
    		\item[(ii)] $a+1\equiv b+1\equiv 0\pmod 2$. Since $\pi'$ avoids $\Oo$, there must be some even numbers lying between $b$ and $a$ in $\pi'$. Suppose $\pi'=\cdots b\cdots ca_1\cdots a_ta\cdots$, where $a_1+1\equiv\cdots\equiv a_t+1\equiv c\equiv 0\pmod 2$, then $c>a>a_t>\cdots>a_1$ due to the avoidance of $\Oo$ and $\eO_d$. Now if $c>b$, in region $R$ we have $x_b<x_c<x_a$, rendering $\pi$ incompatible with $R$, a contradiction. If $c<b$, applying compatibility we see $x_c<x_a$ in region $R$ and $c$ should precede $a$ in $\pi$ as well. This in turn implies that $c$ precedes $b$ in $\pi$. Now we could view $(c,b)$ as a pair in case (iii) below, giving rise to a similar contradiction.
    		\item[(iii)] $a\equiv b+1\equiv 0\pmod 2$. Since $\pi$ avoids $\eO_d$, we find the nearest even number to the left of $b$ and suppose $\pi=\cdots a\cdots cb_1\cdots b_tb\cdots$, where $b_1+1\equiv\cdots\equiv b_t+1\equiv c\equiv 0\pmod 2$, again we have $c>b>b_t>\cdots>b_1$ due to the avoidance of $\Oo$ and $\eO_d$. Applying compatibility, we see $x_c<x_b$ in region $R$ and $c$ should precede $b$ in $\pi'$ as well. This in turn implies that $c$ precedes $a$ in $\pi'$. Now we could view $(a,c)$ as a pair in case (i), giving rise to a similar contradiction.
    	\end{itemize}
    \end{itemize}
    
    Lastly, for $\LIV:=\LIV_n: \cR(\cK_{2n}) \to \GIV_{2n}$, it suffices to define $\LIV_{n+1}$ which completes the following diagram. We remark that the construction of $\LIV$ is in some sense dual to the previous three labelings.
    
    \begin{figure*}[h!]
    	\begin{tikzpicture}[scale=0.8]
    		\draw(0,0) node{$\cR(\cK_{2n+2})$};
    		\draw(5,0) node{$\GIV_{2n+2}$};
    		\draw(0,3) node{$\cR(\cK_{2n})$};
    		\draw(5,3) node{$\GIV_{2n}$};
    		\draw(5.7,1.5) node{IA-IV};
    		
    		\draw(2.5,3.3) node{$\LIV_n$};
    		\draw(2.5,0.3) node{$\LIV_{n+1}$};
    		\draw(-.7,1.5) node{$\proj'$};

    		\draw[>=triangle 45, ->, dashed] (1,0) -- (4.3,0);
    		\draw[>=triangle 45, ->] (1,3) -- (4.3,3);
    		\draw[>=triangle 45, <-] (5,0.5) -- (5,2.5);
    		\draw[>=triangle 45, ->] (0,0.5) -- (0,2.5);
    	\end{tikzpicture}
    \end{figure*}
    
    An even (resp.~odd) factor is said to be \emph{maximal} if it is neighboured by odd (resp.~even) letters. We begin by noting that due to the avoidance of patterns $\Oo_d$ and $\Ee_d$ (the $d=0$ case suffices), every permutation in $\GIV$ can be uniquely decomposed into maximal even and odd factors, such that each factor is monotonously increasing. Fix a given region $R\in\cR(\cK_{2n+2})$, we first project it to a region $\tilde{R}\in\cR(\cK_{2n})$, but the projection employed here is different from the one associated to the previous three algorithms. It is defined as follows and we shall denote it as $\proj'$ to distinguish.
    \begin{align*}
    	\proj':~R &\to \tilde{R} \\
    	\bx=(x_1,x_2,x_3\cdots,x_{2n+3}) &\mapsto \tilde{\bx}=(x_3,\cdots,x_{2n+3}).
    \end{align*}
    
    Now by inductive hypothesis we can find the unique permutation $\LIV_{n}(\tilde{R})=a_1a_2\cdots a_{2n}\in\GIV_{2n}$ and let $\tilde{\pi}:=(a_1+2)(a_2+2)\cdots (a_{2n}+2)$. Next, we insert $1$ and $2$ into $\tilde{\pi}$ with necessary modifications when needed to derive a unique permutation $\pi\in\GIV_{2n+2}$. Finally, we set $\LIV_{n+1}(R):=\pi$. 
    
    Pick any vector $\bx\in R$, we first apply the following boxed {\it insertion algorithm IV} to insert $y=1$ and output the word $w$. Then we insert $2$ into $w$ according to the following two cases.
    \begin{itemize}
    	\item If $x_1<x_2$ in $\bx$, then we find in $w$ the leftmost even number $\ell$ to the right of $1$, and insert $2$ to the immediate left of $\ell$ to get $\pi$. In the case that no such $\ell$ exists, we simply append $2$ to the right end to obtain $\pi:=w2$. 
    	\item If $x_1>x_2$ in $\bx$, then let $\ell$ be the leftmost even number in $w$. If $\ell$ precedes $1$ in $w$ we insert $2$ to the immediate left of $\ell$ to get $\pi$. Otherwise we must have $w_1=1$ since $w$ avoids $\Oo_d$, and we set $\pi:=2w$.
    \end{itemize}
    
    A similar notion of ``bad'' is required. A pair of even integers $(i,j)$ is said to be {\it e-bad} with respect to $\bx$, if $i$ precedes $j$ in $\tilde{\pi}$ but $x_j<x_{1}<x_i$ in $\bx$. We need the following facts about e-bad pairs. 
    
    \begin{claim}\label{cla:bad pair IV}
    	Suppose the even pair $(i,j)$ is e-bad in $\tilde{\pi}$ with respect to $\bx$, then we claim that
    	\begin{enumerate}
    		\item $i<j$;
    		\item for every odd number $k$ that appears between $i$ and $j$, we have $i<k$;
    		\item if $i$ belongs to a maximal even factor $\alpha$, then $i<\max(\alpha)$.
    	\end{enumerate}
    \end{claim}
    
    To see item (1), assume on the contrary that $i>j$, recall that $\tilde{\pi}$ avoids parity patterns $\Ee_d$, hence there exists at least one odd entry between $i$ and $j$. Suppose $\pi=\cdots i \cdots ck_1\cdots k_tj_1\cdots j_sj$, where $c\equiv k_1+1\equiv\cdots\equiv k_t+1\equiv j_1\equiv\cdots j_s\equiv j\equiv 0\pmod 2$, then $c>k_1$ and $k_1<j_1\le j$ due to the avoidance of $\eO$ and $\Ee_d$. But this implies that $x_i<x_{k_1}<x_j$ in both $\bx$ and $\tilde{\bx}$, which contradicts with $x_j<x_1<x_i$. For item (2), if $i>k$ then combining with item (1) we see $k$ is an odd number less than both $i$ and $j$, hence we arrive at the same contradictory inequality $x_i<x_k<x_j$. Finally for item (3), if $i$ and $j$ belong to the same even factor (i.e., all letters sitting between $i$ and $j$ are even), the claim clearly holds since $i<j$ by (1). Otherwise we suppose the even factor $\alpha$ that contains $i$ is followed by an odd letter $k$, which precedes $j$. According to item (2) we must have $i<k$, and $i$ and $k$ cannot be consecutive since $\tilde{\pi}$ avoids $\eO$, hence $i<\max(\alpha)$.
    
    The algorithm below tells us how to remedy all the e-bad pairs while inserting $y=1$ into $\tilde{\pi}$. 
    
    \begin{framed}
    	\begin{center}
    		{\bf Insertion Algorithm IV (IA-IV)}
    	\end{center}
    	\begin{itemize}
    		\item[Initiate] Input $y$, $\tilde{\pi}$, and $\bx$.
    		\item[Step 1] If there are no e-bad pairs in $\tilde{\pi}$ with respect to $\bx$, then let $u:=\varnothing$, $\tilde{w}:=\tilde{\pi}$ and jump to Step 3. Otherwise, find the rightmost $j$ in $\tilde{\pi}$ such that $(i,j)$ forms an e-bad pair with respect to $\bx$ for a certain $i$.
    		\item[Step 2] Find all e-bad pairs ending with $j$, say $(i_1,j),\ldots,(i_m,j)$ with $i_1<\cdots<i_m$. Remove $i_1,\ldots,i_m$ from $\tilde{\pi}$ to get a new word $\tilde{w}$, and form the subword $u:=i_1 i_2 \cdots i_m$ by concatenation.
    		\item[Step 3] Find the rightmost even number $k$ in $\tilde{w}$ such that $x_k<x_1$. Note that if the algorithm goes through Step 2, then $k=j$. Decompose the word $\tilde{w}$ as 
    			$$\tilde{w}=\alpha \beta^{-}\beta^{+}\gamma^{-}k\gamma^{+}\sigma^{-}\sigma^{+}\tau,$$
    		where $\alpha,\beta^{-},\beta^{+},\gamma^{-},\gamma^{+},\sigma^{-},\sigma^{+},\tau$ are all (possibly empty) factors, with $\alpha$ ending with an even letter, $\tau$ beginning with an even letter, and 
    		$$\beta^{-}=\beta^{+}=\sigma^{-}=\sigma^{+}=\mathrm{O},~\gamma^{-}=\gamma^{+}=\mathrm{E},$$
    		as well as
    		$$\max(\beta^{-})<k<\min(\beta^{+}),~\max(\gamma^{-})<k<\min(\gamma^{+}),~\max(\sigma^{-})<i_m<\min(\sigma^{+}).$$
    		There are three subcases to be discussed when $y$ and $u$ are inserted into $\tilde{w}$.
    		\begin{itemize}
    			\item[(i)] If $\gamma^{+}=\sigma^{-}=\varnothing$, we take $w:=\alpha \beta^{-}\gamma^{-}ky (\beta^{+}\sigma^{+}) u\tau$. Note that in this case $\beta^{+}$ and $\sigma^{+}$ cannot both be non-empty\footnotemark, so either $\beta^+\sigma^+=\beta^+$ or $\beta^+\sigma^+=\sigma^+$.
    			\item[(ii)] If $\gamma^{+}=\varnothing$ but $\sigma^{-}\neq\varnothing$, by a similar argument as in (i) (see footnote $\# 4$ below), we see that $\beta^{+}=\varnothing$ to avoid $\Oo_d$ pattern in $\tilde{w}$, and we take $w:=\alpha\beta^{-}\gamma^{-}kyu\sigma^{-}\sigma^{+}\tau$.
    			\item[(iii)] If $\gamma^{+}\neq \varnothing$, we take $w:=\alpha\beta^{-}\gamma^{-}ky\beta^{+}u\gamma^{+}\sigma^{-}\sigma^{+}\tau$.
    		\end{itemize}
    		In the particular case that no such $k$ exists, simply take $w:=y\tilde{w}$.
    		\item[End] Output $w$. 
    	\end{itemize}
    \end{framed}
    \footnotetext{Indeed, if $\beta^{+}\neq\varnothing$ and $\sigma^{+}\neq\varnothing$, then the pair $(\max(\beta^{+}),\min(\sigma^{+}))$ forms an $\Oo_d$ pattern in $\tilde{w}$, a contradiction.}
    
    To finish the proof, one needs to verify the following facts. 
    \begin{itemize}
    	\item $\pi\in\GIV_{2n+2}$. It suffices to show that after each step, the new word obtained still avoids those three parity patterns. There are three scenarios that potential new patterns could arise from. 
    	\begin{enumerate}
    		\item When $i_1,\cdots,i_m$ is removed in step 2. Thanks to Claim ~\ref{cla:bad pair IV} (3), we know that none of $i_1,\cdots,i_m$ could be the ending letter in an even factor, hence no $\eO$, $\Oo_d$, or $\Ee_d$ patterns could be formed.
    		
    		\item  When $u$ and $y=1$ are inserted in step 3, we examine the potential new patterns case-by-case.
    		\begin{itemize}
    			\item[(i)] $\gamma^{+}=\sigma^{-}=\varnothing$. We see $\min(\beta^{-})=\min(\beta^{-}\beta^{+})$ so the ending letter of $\alpha$ and $\min(\gamma^{-}k)$ remain being not an $\Ee_d$ pattern. $\max(\beta^{-})<k$ so $\max(\beta^{-})$ and $y=1$ do not form an $\Oo_d$ pattern. $y=1<i_1$ ensures that $k$ and $i_1$ do not form an $\Ee_d$ pattern. Lastly, when $\beta^{+}\sigma^{+}=\sigma^{+}$ and $\tau\neq\varnothing$, if $\max(\sigma^{+})$ and a letter in $\tau$ form an $\Oo_d$ pattern in $w$, then this occurrence of $\Oo_d$ would have already existed in $\tilde{w}$, a contradiction. When $\beta^{+}\sigma^{+}=\beta^{+}$, then actually $\tau=\varnothing$ and $w$ ends with $u$, thus nothing needs to be further checked.
    			\item[(ii)] $\gamma^{+}=\varnothing$ but $\sigma^{-}\neq\varnothing$. By Claim~\ref{cla:bad pair IV} (1) we have $\max(u)=i_m<k$, so if $i_m$ and the beginning letter of $\tau$ forms an $\Ee_d$ pattern in $w$, then $k$ and the beginning letter of $\tau$ also forms an $\Ee_d$ pattern in $\tilde{w}$, a contradiction. The remaining verifications go through analogously as in case (i).
    			\item[(iii)] $\gamma^{+}\neq \varnothing$. If $\beta^{+}=\varnothing$, the verifications are almost the same as case (ii). If $\beta^{+}\neq\varnothing$, we need to further check the potential pattern arose from the pair $(\max(\beta^{+}),\min(\sigma^{-}\sigma^{+}))$. Suppose this is indeed an $\Oo_d$ pattern, then $\max(\beta^{+})>\max(\gamma^{+})>\min(\sigma^{-}\sigma^{+})$, but this in turn implies that the same pair $(\max(\beta^{+}),\min(\sigma^{-}\sigma^{+}))$ would have formed an $\Oo_d$ pattern in $\tilde{w}$ already, which is a contradiction.
    		\end{itemize}
    		 
    		\item When $2$ is inserted into $w$ after the running of IA-IV. For the two special cases $\pi=w 2$ or $\pi=2 w$, it is readily checked that $\pi\in\GIV_{2n+2}$. Otherwise there are two more cases. If $x_1<x_2$ in $\bx$, then we can write $\pi=\cdots\alpha\beta b\cdots$, where $\alpha$ is an odd factor beginning with $1$, $\beta$ is an even factor beginning with $2<\max(\beta)$, and $b$ is odd. Now if $(\max(\alpha),b)$ becomes an $\Oo_d$ pattern meaning that $\max(\alpha)>\max(\beta)>b$, then $(\max(\alpha),b)$ would have been an $\Oo_d$ pattern already in $w$, a contradiction. A similar argument works for the case when $x_1>x_2$ in $\bx$, and $2$ is inserted to the immediate left of the leftmost even letter in $w$.

    	\end{enumerate}

    	\item The image permutation $\pi\in\GIV_{2n+2}$ is compatible with the given region $R$. The relative orders between $1$ and every $2j$, $1\le j\le n+1$ is compatible with $R$ since we have gotten rid of all the e-bad pairs via IA-IV. In addition, comparing the positions of $i_1,i_2,\ldots,i_m$ in $\tilde{\pi}$ and $\pi$ reveals that they only passed odd numbers that are greater than themselves (thanks to Claim~\ref{cla:bad pair IV} (2)). Furthermore, the odd factor $\beta^{+}$ being moved in Step 3 of IA-IV only passed even numbers that are less than $\min(\beta^{+})$. So the relative orders between $2i-1$ and $2j$ ($2\le i\le j\le n+1$) in $\pi$ remain the same as in $\tilde{\pi}$. Consequently, we can use induction hypothesis to conclude that $\pi$ is indeed compatible with $R$.
    	
    	\item The permutation in $\GIV_{2n+2}$ that is compatible with the given region $R$ is unique. Suppose $\pi,\pi'\in\GIV_{2n+2}$ are two distinct permutations both compatible with $R$. Find a pair $(a,b)$ with $a<b$ such that $a$ precedes $b$ in $\pi$ while $b$ precedes $a$ in $\pi'$. Fix a vector $\bx\in R$ and consider the following four cases. 
    	\begin{itemize}
    		\item[(i)] $a\equiv b\equiv 0\pmod 2$. The permutation $\pi'$ in this case has already been discussed in the proof of Claim~\ref{cla:bad pair IV} (1), thus we have $x_b<x_{k_1}<x_a$ in $\bx$, where $k_1$ is an odd letter less than $b$ and $a$ while sitting between $b$ and $a$. This implies that $\pi$ is incompatible with $R$, a contradiction. 
    		
    		\item[(ii)] $a+1\equiv b+1\equiv 0\pmod 2$. Since $\pi'$ avoids $\Oo_d$, there must be some even numbers lying between $b$ and $a$ in $\pi'$. Suppose $\pi'=\cdots bb_1\cdots b_tc_1\cdots c_s d\cdots a\cdots$, where $b_1+1\equiv\cdots\equiv b_t+1\equiv c_1\equiv\cdots\equiv c_s \equiv d+1\equiv0\pmod 2$, then $b\le b_t<c_s$ due to the avoidance of $\Oo_d$ and $\eO$, thus we have $x_b<x_{c_s}<x_a$ in $\bx$, rendering $\pi$ incompatible with $R$, a contradiction. 
    		
    		\item[(iii)] $a\equiv b+1\equiv 0\pmod 2$. Since $\pi$ avoids $\eO$, there must be some numbers lying between $a$ and $b$ in $\pi$. Assume that $\pi=\cdots aa_1\cdots a_tc\cdots b\cdots$, where $a_1\equiv\cdots\equiv a_t\equiv c+1\equiv0\pmod 2$. Notice that $a=a_t$ and $c=b$ cannot both happen. In addition, we have $a_t>c$ due to the avoidance of $\eO$, thus in $\bx$ we have $x_{a_t}<x_c$, implying that $a_t$ precedes $c$ in $\pi'$. There are five subcases to consider. 
    		\begin{enumerate}[label=(\alph*)]
    			\item If $a=a_t$ and $c\neq b$, then $a>c$ and thus $x_a<x_c$ in $\bx$, implying that $a$ (hence $b$ as well) precedes $c$ in $\pi'$. Therefore we can treat $(c,b)$ as a pair in case (ii) and deduce similar contradiction. 
    			\item If $c=b$ and $a\neq a_t$, then $a_t>b$ and thus $x_{a_t}<x_b$ in $\bx$, implying that $a_t$ precedes $b$ (hence $a$ as well) in $\pi'$. Therefore we can treat $(a,a_t)$ as a pair in case (i) and deduce similar contradiction. For the remaining cases we may assume that $a\neq a_t$ and $c\neq b$.
    			\item If $a_t$ precedes $a$ in $\pi'$, then $(a,a_t)$ is a pair in case (i) and we arrive at similar contradiction.
    			\item If $b$ precedes $c$ in $\pi'$ then $(c,b)$ is a pair in case (ii) and we arrive at similar contradiction.
    			\item If none of the above four cases happen, we see the four letters $a,a_t,c,b$ appear in $\pi'$ in that order, which contradicts with our choice of $\pi'$.
    		\end{enumerate} 

    		\item[(iv)] $a+1\equiv b\equiv 0\pmod 2$. This case cannot happen, which can be showed in a similar fashion as in the proof of Theorem~\ref{thm:label-I} (iv).
    	\end{itemize}
    \end{itemize}
\end{proof}
 
 \begin{example}
 	To illustrate the bijections constructed in the proof of Theorem~\ref{thm:three more labelings}, as well as to make a comparison with the first model $\GI$, we label the region $R$ from Example~\ref{ex-I} with $\LII$, $\LIII$, and $\LIV$, respectively. 
 	
 	For the bijection $\LII$, we also define $\tilde{R}_i=\proj^{8-i}(R)$ for $i=0,\ldots,7$ and $\tilde{\pi}^{(j)}=\LII_{j}(\tilde{R}_{j-1})$ for $j=1, \ldots, 8$. One reads from the Hasse diagram in Fig.~\ref{Hasse diagram} that $\tilde{R}_0=\{\bx\in\bR^3:x_1-x_2<0\}$, so initially $\tilde{\pi}^{(1)}=\LII_1(\tilde{R}_0)=12$, and we will insert $3$ and $4$. We see $x_3<x_4$ in the diagram so we set $y=4$ and start IA-II. There are no bad pairs in $\tilde{\pi}^{(1)}$ with respect to $\bx$ and there exists no $k$ such that $x_{k}>x_4$, so 
 	$$w=124,\text{ and }\tilde{\pi}^{(2)}=3124.$$
 	Next, we insert $5$ and $6$. The diagram shows $x_5>x_6$, so we set $y=65$, there are no bad pairs and $3$ is the leftmost $k$ in $\tilde{\pi}^{(2)}$ such that $x_k>x_6$, so
 	$$k=3,~\tilde{\pi}^{(3)}=w=653124.$$
 	We continue to insert $7$ and $8$. The diagram shows $x_7<x_8$, so $y=8$. Now $(3,1)$ is a bad pair since $3$ precedes $1$ in $\tilde{\pi}^{(3)}$ while the diagram indicates that $x_1<x_8<x_3$. Consequently, applying IA-II we derive
 	$$i_1=k=3,~u=3,~z=4,~\tilde{w}=651234,~w=6512834,\text{ and }\tilde{\pi}^{(4)}=76512834.$$
 	The remaining numbers can be inserted analogously using IA-II so we omit the details. The key parameters for each round of IA-II are listed below to reflect the incremental changes. 
 	\begin{align*}
 		& y=(10)9,~i_1=k=7,~u=7,~z=1,~\tilde{w}=65712834, \text{ and }\tilde{\pi}^{(5)}=w=65(10)9712834.\\
 		& y=(12),~i_1=k=9,~u=9,~\tilde{w}=65(10)7128349,\\
 		&\qquad \qquad \qquad w=65(10)712834(12)9,\text{ and }\tilde{\pi}^{(6)}=(11)65(10)712834(12)9.\\
 		& y=(14)(13),~i_1=k=(11),~u=(11),~z=(12),~\tilde{w}=65(10)712834(11)(12)9, \text{ and}\\
 		& \qquad \qquad \qquad \tilde{\pi}^{(7)}=w=65(10)712834(14)(13)(11)(12)9.\\
 		& y=(16),~i_1=k=(13),~u=(13),~\tilde{w}=65(10)712834(14)(11)(12)9(13),\\
 		& \qquad w=65(10)712834(14)(11)(12)9(16)(13),\text{ and }\tilde{\pi}^{(8)}=(15)65(10)712834(14)(11)(12)9(16)(13).\\
 		& y=(18)(17),~i_1=k=(15),~u=(15)7,~z=8, ~\tilde{w}=65(10)12(15)7834(14)(11)(12)9(16)(13), \text{ and }\\
 		& \qquad \LII_9(R)=\tilde{\pi}^{(9)}=w=65(10)12(18)(17)(15)7834(14)(11)(12)9(16)(13).
 	\end{align*}
 	
 	For the bijection $\LIII$, after $8$ rounds of running IA-III we derive that 
 	$$\tilde{\pi}=(15)65(10)1783(14)(11)(12)9(16)(13)42.$$
 	Now we insert $17$ and $18$. According to the diagram, $(15,1)$ and $(15,5)$ are the bad pairs in $\tilde{\pi}$, so applying IA-III we derive 
 	$$j=k=1,~u=(15),~v=7,~z=8, 
 	~w=65(10)1(18)7(15)83(14)(11)(12)9(16)(13)42.$$
 	The diagram shows $x_{17}>x_{18}$, so we insert $17$ before the leftmost even number to the right of $18$, which is $8$. Consequently, we deduce that
 	$$\LIII_9(R)=\pi=65(10)1(18)7(15)(17)83(14)(11)(12)9(16)(13)42.$$ 
 	
 	For the bijection $\LIV$, after $8$ rounds of running IA-IV we derive that 
 	$$\tilde{\pi}=65(10)(18)7834(14)(11)(12)9(15)(16)(13)(17).$$
 	Now we insert $1$ and $2$. There are no e-bad pairs in $\tilde{\pi}$ and $k=(10)$ with $\gamma^{+}=(18)\neq\varnothing$, so we apply the case (iii) in Step 3 (note that $\beta^{+}=\varnothing$) to output $$w=65(10)1(18)7834(14)(11)(12)9(15)(16)(13)(17).$$ 
 	The diagram shows $x_1<x_2$, so we insert $2$ to arrive at
 	$$\LIV_9(R)=\pi=65(10)12(18)7834(14)(11)(12)9(15)(16)(13)(17).$$
 \end{example}

\subsection{The proof of Theorem~\ref{thm:GIS}}
When the enumeration of permutations in $\GI$ is refined with respect to their first letter, the triangular array in Table~\ref{Tab:GI} is produced. For convenience, we make the convention that $\Gi_{2n,k}=0$ for all $k<1$ and $k>2n$. In this subsection, we prove some properties for the entries in Table~\ref{Tab:GI} which includes Theorem~\ref{thm:GIS}.

The following relation linking Genocchi numbers with tangent numbers is well known; see~\cite{HL18} for a combinatorial approach and \cite{ZLZ25} for an elementary proof that relies on an identity and uses induction.
\begin{theorem}
Let the tangent number $T_{2n+1}$ be defined as the coefficient of $x^{2n+1}$ in the Taylor expansion $\tan(x)=\sum_{n\ge 0}T_{2n+1}x^{2n+1}/(2n+1)!$, then for all $n\ge 0$ we have
\begin{align*}
g_{n+1} = \frac{(n+1)T_{2n+1}}{2^{2n}},
\end{align*}
and $g_n$ is always an odd integer.
\end{theorem}

The first two columns in Table~\ref{Tab:GI} are both Genocchi numbers and thus are odd in view of the result above. We begin with an observation that determines the parity for the remaining entries in Table~\ref{Tab:GI}. For $n\ge 2$ and $k\ge 3$, let us denote
\begin{align}
	\tGI_{2n,k}:=\{\pi\in\GI_{2n,k}:\pi_{2n}\neq 2\}, \text{ and }\cGI_{2n,k}:=\{\pi\in\GI_{2n,k}:\pi_{2n}=2\}. \label{bisection}
\end{align}

\begin{proposition}\label{prop:div}
	For all $n\ge 1$, we have 
	\begin{enumerate}[label=(\roman*)]
		\item There exists a bijection for every $k\ge 3$: 
		$$\eta : \tGI_{2n,k} \to \cGI_{2n,k}.$$ 
		In particular, $\Gi_{2n,k}$ is even for any $k\ge 3$.
		\item $\Gi_{2n,1}=\Gi_{2n,2}$, which is \eqref{id:1=2} in Theorem~\ref{thm:GIS}.
	\end{enumerate}
\end{proposition}

\begin{proof}
	The bijection is defined as $$\eta: \pi_1\cdots \pi_j 21 \pi_{j+3}\cdots \pi_{2n}\mapsto \pi_1\cdots \pi_j 1 \pi_{j+3}\cdots \pi_{2n}2.$$
	This is indeed a bijection since every permutation not ending in $2$ must contain $21$ as a factor to avoid patterns $\eE$ and $\eO$. This proves (i).

	For the particular case that $\pi_1\cdots \pi_j=\varnothing$, we see that $\eta$ restricts to a bijection between $\GI_{2n,1}$ and $\GI_{2n,2}$. Consequently we have $\Gi_{2n,1}=\Gi_{2n,2}$.
\end{proof}


Before we prove the rest of identities in Theorem~\ref{thm:GIS}, the following notation is introduced for our convenience.
\begin{Def}{\cite[Definition~1.0.9]{kit11}}
	The {\it reduced form} of a permutation $\pi$ on a set $\{j_1,j_2,\ldots,j_r\}$ where
	$j_1<j_2<\cdots<j_r$, is the permutation $\pi'$ obtained by renaming the letters in $\pi$ so that $j_i$ is renamed $i$ for all $i\in \{1,\ldots,r\}$. In other words, to find the reduced form of a permutation $\pi$ on $r$ elements, we replace the $i$-th smallest letter of $\pi$ by $i$, for $i=1,\ldots,r$. We let $\re(\pi)$ denote the reduced form of $\pi$.
\end{Def}
For example $\re(435)=213$, and $\re(2165)=2143$.


\begin{proof}[Proof of Theorem~\ref{thm:GIS}]
	Now that \eqref{id:1=2} is already shown in Proposition~\ref{prop:div}, we prove the remaining four identities in items (i)--(iv), respectively. Throughout we let $\pi =\pi_1 \cdots \pi_{2n} \in \GI_{2n}$ be a given permutation.
	\begin{enumerate}[label=(\roman*)]		
		\item By Proposition~\ref{prop:div} (i), to prove \eqref{E_3}, it suffices to construct a bijection
		$$f: \tGI_{2n,3}:=\{\pi\in\GI_{2n,3}: \pi_{2n}\neq 2\}\to \GI_{2n,1}\bigcup\GI_{2n-2,1}.$$
		We construct the image $f(\pi)$ according to the following two subcases. 
		\begin{itemize}
			\item If $\pi = 3 \pi_2 \cdots  \pi_{j} 21 \pi_{j+3} \cdots \pi_{2n}$ whith $\pi_{2n}\neq 1$, then $\pi_{2n} \ge 4$. We let
			$$f(\pi) := 1\pi_{j+3}\cdots \pi_{2n}3 \pi_2 \cdots \pi_{j}2\in \GI_{2n,1}.$$
			Conversely, since permutations in $\GI_{2n,1}$ always end with $2$, each of them has a unique preimage in $\tGI_{2n,3}$. 
			\item If $\pi =3 \pi_2  \cdots \pi_{2n-2}21$, we let $f(\pi):=\mathrm{red}(3 \pi_2 \cdots \pi_{2n-2})$. Conversely, each permutation $\sigma=1\sigma_2\cdots\sigma_{2n-2}\in\GI_{2n-2,1}$ has a unique preimage $\pi:=3(\sigma_2+2)\cdots(\sigma_{2n-2}+2)21$ in $\tGI_{2n,3}$.
		\end{itemize}
		The mapping $f$ as described above is clearly a bijection, which completes the proof of \eqref{E_3}.
		
		\item Note that every $\pi \in \GI_{2n,2n}$ has the form $\pi = (2n)(2n-1)\pi_3\cdots \pi_{2n}$, where $\pi_3$ can be any number from $[2n-2]$. Removing the prefix $(2n)(2n-1)$ from $\pi $ results in a permutation $\pi'\in\GI_{2n-2}$. Conversely, appending $(2n)(2n-1)$ to the left end of $\pi'$ yields a permutation in $\GI_{2n,2n}$. This correspondence is clearly a bijection hence we have \eqref{En2n=sum of above line}.
		
		\item We construct the following bijection to establish \eqref{eq:GI_2n,2n-1}:
		$$g: \GI_{2n}\setminus\GI_{2n,2n-1} \to \GI_{2n,2n-1}.$$
		Let $\pi=\pi_1\cdots (2n)(2n-1)\pi_j \cdots\pi_{2n} \in \GI_{2n}\setminus\GI_{2n,2n-1}$. We derive its image $g(\pi)$ by advancing $2n-1$ to the left end. This is well defined since $2n>\pi_{j}$, and either $2n-1>\pi_1<2n$ or $2n-1<\pi_1=2n$. The inerse operation $g^{-1}$ is clearly moving $2n-1$ to the immediate right of $2n$.
		
		\item We give here an algebraic proof of \eqref{id:GI and S}, which relies on the following recursion for $S_{2n,k}$. This proof can be turned into a recursively defined bijection that we shall present in the next section. 
		\begin{align}\label{rec:S}
			S_{2n,n-k}=S_{2n,n+1-k}+\sum_{i=k}^{n-1}S_{2n-2,n-i}.
		\end{align}
		Assuming Theorem~\ref{En-even}, we have $$\Gi_{2n,2k+2}=\Gi_{2n,2k}+\sum\limits_{i=k}^{n-1} \Gi_{2n-2,2i},$$
		this recurrence relation combined with \eqref{rec:S} and the initial values $S_{2,1}=\Gi_{2,2}=1$ will establish \eqref{id:GI and S} by induction. It remains to prove \eqref{rec:S}. Indeed, by repeatedly applying the generation rule \eqref{Seidel rule}, we see that $S_{2n,k}=\sum_{i\ge k}S_{2n-1,i}$ and $S_{2n+1,k}=\sum_{i\le k}S_{2n,i}$, so we deduce that
		\begin{align*}
			S_{2n,n-k}-S_{2n,n+1-k}&=\sum_{i\ge n-k} S_{2n-1,i}-\sum_{i\ge n+1-k} S_{2n-1,i}
			=S_{2n-1,n-k}\\
			&=\sum_{i\le n-k} S_{2n-2,i}
			=\sum_{n-i\geq k} S_{2n-2,i}
			=\sum_{i\geq k} S_{2n-2,n-i}.
		\end{align*}
		This is \eqref{rec:S} and the proof is now completed.
	\end{enumerate} 
\end{proof}

\section{Three bijections}\label{sec:bijs}
Both kinds of Genocchi numbers have more than one combinatorial interpretations, as can be gleaned from Definition~\ref{def:five perm}. This raises natural bijective questions. In this section, we present three bijections pertaining to our first permutation model $\GI_{2n}$, the Dumont permutations of the third kind $\D^{\mathrm{III}}_{2n}$, and the collapsed permutations $\CO_{2n}$, which were first introduced by Ayyer, Hathcock, and Tetali~\cite{AHT22} and can be defined as follows. 

\begin{Def}[{\cite[Definition~4.1]{AHT22}}] \label{collapsed permutation def}
	A permutation $\pi\in\S_{2n}$ is said to be \emph{collapsed}, if for every $k\in [2n]$ we have
	\begin{align} \label{defn:collapse}
		1+ \lfloor{k/2}\rfloor \leq \pi^{-1}(k) \leq n+\lfloor k/2 \rfloor.
	\end{align}
	We denote the set of collapsed permutations of length $2n$ by $\CO_{2n}$ and its cardinality by $\Co_{2n}$.
\end{Def}

For example, the first three sets of collapsed permutations are listed below.
\begin{align*}
		&\CO_{2}=\{12\},\quad \CO_{4}=\{1234, 1324\}, \\
		&\CO_{6}=\{123456, 123546, 124356, 125346, 132456, 132546, 134256, 135246\}.
\end{align*}

Our first bijection is constructed recursively, building on combinatorial interpretations of the two recursions in \eqref{id:GI even rec} and \eqref{rec:S}, respectively. Our second and third bijections can be combined to give a direct bijective proof of the following result.
\begin{theorem}
For all $n\ge 0$, we have
\begin{align}\label{id:Co and GI}
\Co_{2n+2}=h_n=\Gi_{2n}.
\end{align}
\end{theorem}
\begin{proof}
The first equality $\Co_{2n+2}=h_n$ was already derived in \cite[Theorem~4.3]{AHT22}, while the second equality $\Gi_{2n}=h_n$ follows from Theorem~\ref{thm:label-I}, Theorem~\ref{thm:hetyei}, and Proposition~\ref{prop:K2n}.
\end{proof}

\subsection{A proof of Theorem~\ref{En-even} and the first bijection}
This subsection contains three main results. Firstly, we construct a bijection $\varphi$ that proves Theorem~\ref{En-even}; secondly, we rederive \eqref{rec:S} via a second bijection $\phi$; finally, these two bijections are combined to generate recursively a bijection $\Phi$ between $\GI_{2n,2k}$ and 
$$\D^{\mathrm{III}}_{2n,2k}:=\{\pi\in\D^{\mathrm{III}}_{2n}: \text{the leftmost even letter in $\pi$ is $2k$}\},$$
for all $1\le k\le n$. This latter set was introduced by Burstein, Josuat-Verg\`es, and Stromquist~\cite[Definition~5.1]{BJS10}. They also showed in \cite[Theorem~5.3]{BJS10} that $|\D^{\mathrm{III}}_{2n,2k}|=S_{2n,n+1-k}$. In view of this, the bijection $\Phi:\GI_{2n,2k}\to\D^{\mathrm{III}}_{2n,2k}$ indeed provides a bijective approach to \eqref{id:GI and S}, as alluded to in the proof of Theorem~\ref{thm:GIS} item (iv).

\begin{proof}[Proof of Theorem~\ref{En-even}]
	For all $n\ge 2$ and $0\le k < n$, we aim to construct a bijection  $$\varphi:=\varphi_{n,k}: \GI_{2n,2k+2} \to \GI_{2n,2k}\bigcup\left(\bigcup_{i=k}^{n-1}\GI_{2n-2,2i}\right),$$ 
	which readily implies \eqref{id:GI even rec} by taking cardinalities.
	
	For $k=0$, we need to show that $\GI_{2n,2}$ is in bijection with $\bigcup_{i=1}^{n-1}\GI_{2n-2,2i}$. Clearly $\pi \in \GI_{2n,2}$ if and only if $\pi =21 \pi_3 \cdots \pi_{2n}$ with $\pi_3\ge 4$ even, hence we let $\varphi_{n,0}(\pi):=\re(\pi_3\cdots \pi_{2n})\in\GI_{2n-2,\pi_3-2}\subset\bigcup_{i=1}^{n-1}\GI_{2n-2,2i}$. Conversely, every permutation in $\GI_{2n-2}$ beginning with an even number can be turned into a permutation in $\GI_{2n,2}$ by increasing each number by $2$, and appending $21$ to its left end. So we see $\varphi_{n,0}$ maps $\GI_{2n,2}$ bijectively onto $\bigcup_{i=1}^{n-1}\GI_{2n-2,2i}$, as desired.
   
	When $k\ge 2$, recall the definition \eqref{bisection} and the bisection $\GI_{2n,2k}=\tGI_{2n,2k}\uplus\cGI_{2n,2k}$ from Proposition~\ref{prop:div} (i). For $k=1$, we need to show that $\GI_{2n,4}$ is in bijection with $\GI_{2n,2}\bigcup\left(\bigcup_{i=1}^{n-1}\GI_{2n-2,2i}\right)$. In view of the previous case with $k=0$, it suffices to construct $\varphi_{n,1}$ so that
	$$\cGI_{2n,4}\quad \underset{\text{\tiny 1 to 1}}{\xtwoheadrightarrow{\varphi_{n,1}}} \quad \bigcup_{i=1}^{n-1}\GI_{2n-2,2i},$$
	and then for any permutation $\sigma\in\tGI_{2n,4}$, let $\varphi_{n,1}(\sigma):=\varphi_{n,0}^{-1}(\varphi_{n,1}(\eta(\sigma)))$, where $\eta$ is the bijection used in the proof of Proposition~\ref{prop:div} (i). Now we construct the image under $\varphi_{n,1}$ of a given permutation $\pi=4\pi_2\cdots \pi_{2n-1}2 \in \cGI_{2n,4}$. There are three cases, each of which is seen to be invertible.
	\begin{enumerate}
		\item If $\pi_2=1$, then let
		$$\pi=41\pi_3\cdots \pi_{m-1}3\pi_{m+1}\cdots 2 \mapsto \varphi(\pi):=\re(\pi_3\cdots \pi_{m-1}21\pi_{m+1}\cdots \pi_{2n-1}),$$
		where $\varphi(\pi)\in \bigcup_{i=2}^{n-1} \GI_{2n-2,2i}$ does not end in $2$.
		\item If $\pi_2=3$ and $\pi_3 \geq6$ is even, then let
		$$\pi=43\pi_3\cdots \pi_{m-1}1\pi_{m+1}\cdots 2 \mapsto \varphi(\pi):=\re(\pi_3\cdots \pi_{m-1}1\pi_{m+1}\cdots 2),$$
		where $\varphi(\pi)\in \bigcup_{i=2}^{n-1} \GI_{2n-2,2i}$ ends in $2$.
		\item If $\pi_2=3$ and $\pi_3=1$, then let
		$$\pi=431\pi_4\cdots \pi_{2n-1}2 \mapsto \varphi(\pi):=\re(21\pi_4\cdots \pi_{2n-1})\in \GI_{2n-2,2}.$$
	\end{enumerate}
	
	For $1<k<n$, since each of the subsets $\GI_{2n,2k+2},\GI_{2n,2k},\GI_{2n-2,2k},\ldots,\GI_{2n-2,2n-2}$ has the bisection, we only construct $\varphi_{n,k}$ so that
	$$\cGI_{2n,2k+2}\quad \underset{\text{\tiny 1 to 1}}{\xtwoheadrightarrow{\varphi_{n,k}}} \quad \cGI_{2n,2k}\bigcup\left(\bigcup_{i=k}^{n-1}\cGI_{2n-2,2i}\right),$$
	and for those permutations $\sigma\in\tGI_{2n,2k+2}$, simply let
	$$\varphi_{n,k}(\sigma):=\eta^{-1}(\varphi_{n,k}(\eta(\sigma))).$$

	Now take a permutation $\pi=(2k+2)\pi_2\cdots\pi_{j-1}(2k)\pi_{j+1}\cdots 2 \in \cGI_{2n,2k+2}$, we are going to construct its image $\varphi(\pi)$. There are three cases according to the value of $\pi_2$. Note that $\pi_2<2k+2$ and $\pi_{j+1}<2k$ due to the avoidacne of $\eE$ and $\eO$.
	\begin{enumerate}[label=(\arabic*)]
		\item If $\pi_2\le 2k-1$, we get $\varphi(\pi)$ by exchanging $2k$ and $2k+2$ in $\pi$, thus 
		$$\varphi(\pi):= (2k)\pi_2\cdots\pi_{j-1}(2k+2)\pi_{j+1}\cdots 2\in \cGI_{2n,2k} \text{ with $\pi_{j+1}<2k$. }$$

		\item If $\pi_2=2k$, suppose $\pi=(2k+2)(2k)\cdots a\beta(2k+1)c\cdots 2$, where $a<2k+1$ and $\beta$ is a factor consisting of letters greater than $2k+2$. Note that $2k$ cannot be an ascent bottom while $2k+1$ cannot be an ascent top, so $a\neq 2k$ exists and $\beta\neq\varnothing$. We consider two subcases according to the value of $c$.
		\begin{enumerate}
			\item If $c>2k+1$ is even, we move $2k+2$ from the first place to the front of $2k+1$, i.e., we let 
			$$\varphi(\pi):=(2k) \cdots a\beta(2k+2)(2k+1)c\cdots 2 \in \cGI_{2n,2k}.$$
			\item If $c<2k+1$, we move $2k+2$ from the first place to the front of $2k+1$ and swap $(2k+2)(2k+1)$ with the factor $\beta$, i.e., we let 
			$$\varphi(\pi):=(2k)\cdots a(2k+2)(2k+1)\beta c\cdots 2\in\cGI_{2n,2k}.$$
		\end{enumerate}

		\item If $\pi_2=2k+1$, then $\pi =(2k+2)(2k+1)\pi_3\cdots \pi_{j-1}(2k)\pi_{j+1}\cdots 2$. The discussion further splits into two subcases according to the value of $\pi_3$.
		\begin{enumerate}
			\item If $\pi_3\ge 2k$ is even, then we let
			$$\varphi(\pi):=\re(\pi_3\cdots \pi_{j-1}(2k)\pi_{j+1}\cdots 2)\in \bigcup_{i=k}^{n-1}\cGI_{2n-2,2i}.$$
			\item If $\pi_3<2k$, then we swap $(2k+2)(2k+1)$ with $(2k)$, i.e., we let
			$$\varphi(\pi):=(2k)\pi_3\cdots \pi_{j-1}(2k+2)(2k+1)\pi_{j+1}\cdots 2 \in\cGI_{2n,2k} \text{ with $\pi_{j+1}<2k$.}$$
		\end{enumerate}
	\end{enumerate}
	To finish the proof, one needs to show that $\varphi$ is indeed a bijection. Clearly $\varphi$ is injective and invertible when restricted to each subcase. Furthermore, basing on the characterizations of $\varphi(\pi)$ summarized case-by-case in the table below, one sees that the images obtained in all subcases are mutually disjoint and cover the entire set $\cGI_{2n,2k}\bigcup\left(\bigcup_{i=k}^{n-1}\cGI_{2n-2,2i}\right)$.

	\begin{table*}[!ht]
		\renewcommand{\arraystretch}{1.5}
		\centering
		\begin{tabular}{c|c}
			$\pi$ & $\varphi(\pi)$ \\ \hline
			case (1) & in $\cGI_{2n,2k}$ without factor $(2k+2)(2k+1)$ \\ 
			case (2-a) & in $\cGI_{2n,2k}$ with factor $x(2k+2)(2k+1)y$, where $x>2k+2$ and $y>2k+2$ \\ 
			case (2-b) & in $\cGI_{2n,2k}$ with factor $x(2k+2)(2k+1)y$, where $x<2k+2$ and $y>2k+2$ \\ 
			case (3-a) & in $\bigcup_{i=k}^{n-1}\cGI_{2n-2,2i}$ \\ 
			case (3-b) & in $\cGI_{2n,2k}$ with factor $(2k+2)(2k+1)y$, where $y<2k+2$   
		\end{tabular}
	\end{table*}
\end{proof}

We list below the one-to-one correspondences between $\GI_{6,6}$ and $\GI_{6,4}\bigcup\GI_{4,4}$, and between $\GI_{6,4}$ and $\GI_{6,2}\bigcup\GI_{4,2}\bigcup\GI_{4,4}$, that are determined by the bijection $\varphi$, which are respectively the cases with $(n,k)=(3,2)$ and $(n,k)=(3,1)$ in the proof above.
	\begin{table}[H]
		\renewcommand{\arraystretch}{1.4}
		\centering
		\begin{tabular}{c|c||c|c}
			$\GI_{6,6}$ & $\GI_{6,4}\bigcup \GI_{4,4}$ & $\GI_{6,4}$ & $\GI_{6,2}\bigcup\GI_{4,2}\bigcup\GI_{4,4}$\\ \hline
			651432 & 416532 & 431652 & 2143 \\
			652143 & 421653 & 432165 & 214365 \\
			653142 & 431652 & 436512 & 4312 \\
			653214 & 432165 & 436521 & 216534 \\
			653412 & 436512 & 421653 & 216543 \\ 
			653421 & 436521 & 416532 & 4321 \\ 
			654312 & 4312 & & \\
			654321 & 4321 & & \\
		\end{tabular}
		\caption{The correspondences $\GI_{6,6}\xrightarrow{\varphi}\GI_{6,4}\bigcup \GI_{4,4}$ and $\GI_{6,4}\xrightarrow{\varphi}\GI_{6,2}\bigcup\GI_{4,2}\bigcup\GI_{4,4}$}
	\end{table}

Next, we shift our attention to $\DIII$, the set of Dumont permutations of the third kind. We begin with an observation that parallels Proposition~\ref{prop:div} (i). Note that every permutation in $\DIII$ begins with the letter $1$, let us introduce
$$\tD_{2n,2k}:=\{\pi\in\DIII:\pi_2\neq 3\}, \text{ and } \cD_{2n,2k}:=\{\pi\in\DIII:\pi_2= 3\}.$$
Then it is easy to see that the following mapping is a bijection for every $k\ge 2$:
\begin{align*}
\tau:  \cD_{2n,2k} &\to \tD_{2n,2k} \\
13 \pi_3 \cdots \pi_{i-1}(2k)\pi_{i+1} \cdots \pi_{j-1} 2\pi_{j+1} \cdots &\mapsto 1\pi_3 \cdots \pi_{i-1}(2k)\pi_{i+1} \cdots \pi_{j-1} 23 \pi_{j+1} \cdots.
\end{align*}
In particular, $\tau$ is surjective since the letter that precedes $3$ must be $2$ if it is not $1$, since $3$ cannot be a descent bottom. Therefore we have the following similar bisection when $k\ge 2$:
\begin{align}\label{DIII-bisec}
\DIII_{2n,2k}=\cD_{2n,2k}\uplus\tD_{2n,2k}.
\end{align}



Recall that $|\DIII_{2n,2k}|=S_{2n,n+1-k}$, so the following result could be viewed as a bijective approach to \eqref{rec:S}. We note in passing that a comparable result for $\D^{\mathrm{I}}_{2n}$ was derived by Burstein, Josuat-Verg\`es, and Stromquist~\cite[Theorem~5.4]{BJS10}.

\begin{theorem}\label{DIII-even}
	For all $n\geq 2$, $n>k\geq0$, there exists a bijection
	$$\phi:=\phi_{n,k}: \DIII_{2n,2k+2} \to \DIII_{2n,2k}\bigcup\left(\bigcup_{i=k}^{n-1}\DIII_{2n-2,2i}\right).$$
\end{theorem}

\begin{proof}
	
	For $k=0$, it suffices to notice that $\pi \in \DIII_{2n,2}$ if and only if $\pi =123\pi_4 \cdots \pi_{2n}$, and let $\phi_{n,0}(\pi):=\re(3 \pi_4 \cdots \pi_{2n})\in \bigcup_{i=1}^{n-1}\DIII_{2n-2,2i}$. Moreover, this operation is clearly invertible, so $\phi_{n,0}$ maps $\DIII_{2n,2}$ bijectively onto $\bigcup_{i=1}^{n-1}\DIII_{2n-2,2i}$.
	
	For $k=1$, we first note that any $\pi \in \cD_{2n,4}$ must have $134$ as its prefix, and we let $\phi_{n,1}(\pi):=\re(1\pi_4\cdots\pi_{2n})$ be its image, which is in $\bigcup_{1\le i\le  n-1}\DIII_{2n-2,2i}$. Note that $\phi_{n,1}(\pi)\in\DIII_{2n-2,2}$ if and only if $\pi_4=2$. Next for any $\sigma\in\tD_{2n,4}$ we let $\phi_{n,1}(\sigma):=\phi_{n,0}^{-1}(\phi_{n,1}(\tau^{-1}(\sigma)))$ be its image, relying on the previous case with $k=0$. Together we see $\phi_{n,1}$ indeed maps $\DIII_{2n,4}=\cD_{2n,4}\uplus\tD_{2n,4}$ bijectively onto $\DIII_{2n,2}\bigcup\left(\bigcup_{1\le i\le n-1}\DIII_{2n-2,2i}\right)$.

	
	For $1<k<n$, all subsets involved enjoy a bisection like \eqref{DIII-bisec}, so again we only need to construct $\phi_{n,k}$ such that
	$$\cD_{2n,2k+2}\quad \underset{\text{\tiny 1 to 1}}{\xtwoheadrightarrow{\phi_{n,k}}} \quad \cD_{2n,2k}\bigcup\left(\bigcup_{i=k}^{n-1}\cD_{2n-2,2i}\right),$$
	and then for any $\sigma\in\tD_{2n,2k+2}$, we let 
	$$\phi_{n,k}(\sigma):=\tau(\phi_{n,k}(\tau^{-1}(\sigma))).$$

	Now given a permutation $\pi\in\cD_{2n,2k+2}$, we construct its image $\phi(\pi)$, in the following three cases, by conditioning on the position of letter $(2k+1)$. Suppose 
	$$\pi=13 \cdots a (2k+2)\beta b\cdots c\gamma (2k) d \cdots,$$
	where the prefix $13\cdots a$ contains only odd letters, the factor $\beta$ (possibly empty) contains only letters $<2k$, the factor $\gamma$ (possibly empty) contains only letters $>2k+2$, $b>2k$, and $c<2k+2$. Note further that a nonempty $\beta$ begins with an even, while a nonempty $\gamma$ ends with an even.
	\begin{enumerate}[label=(\arabic*)]
		\item If neither $a$ nor $d$ equals $2k+1$, then swapping $2k+2$ and $2k$ results in a permutation still in $\DIII$, i.e., we let
		$$\phi(\pi):= 13\cdots a (2k)\beta b\cdots c\gamma (2k+2) d \cdots \in \cD_{2n,2k}.$$
		Notice that neither $(2k)(2k+1)$ nor $(2k+1)(2k+2)$ is a factor of $\phi(\pi)$.
		
		\item If $d=2k+1$, then $d$, being odd, must be followed by a letter (if any), say $e>2k+1$. We switch $(2k+2)\beta$ with $(2k)(2k+1)$, i.e., we let 
		\begin{align*}
			\phi(\pi):= 13\cdots a (2k)(2k+1) b\cdots c\gamma (2k+2)\beta e\cdots\in\cD_{2n,2k}.
		\end{align*}

		\item If $a=2k+1$, then depending on whether the factor $\beta$ is empty or not, we have two subcases.
		\begin{enumerate}
			\item If $\beta=\varnothing$, then we take the reduced form after removing $(2k+1)(2k+2)$, i.e, we let
			$$\phi(\pi):=\re(13\cdots b\cdots c\gamma(2k)d\cdots).$$
			Note that the factor $b\cdots c\gamma$ is either empty (in which case $2k+2$ precedes $2k$ in $\pi$), hence $\phi(\pi)\in\cD_{2n-2,2k}$; or it contains\footnote{Indeed, when $\beta=\varnothing$ and $b\cdots c\gamma\neq\varnothing$, going from $2k+2$ to $2k$ in $\pi$ one must pass at least one descent, then the descent top of the first descent is an even number larger than $2k+2$.} at least one even number and its leftmost even number is larger than $2k+2$, hence $\phi(\pi)\in\bigcup_{k+1\le i\le n-1}\cD_{2n-2,2i}$.

			\item If $\beta\neq\varnothing$, we let
			$$\phi(\pi):= 13\cdots (2k)\gamma\beta b\cdots c(2k+1)(2k+2)d\cdots\in\cD_{2n,2k}.$$

			Note that $\beta$ being nonempty ensures that $2k$ is not followed by $2k+1$ in $\phi(\pi)$.
		\end{enumerate}
	\end{enumerate}	
	Employing the constraints on factors $\beta,\gamma$, and letters $b,c$, it should be clear how to invert $\phi$ in each of the subcases. Furthermore, analogous to the proof of Theorem~\ref{En-even}, we characterize below the images $\phi(\pi)$ for all subcases of $\pi$, which should assist the reader in confirming that $\phi$ is indeed a bijection in the case of $1<k<n$.
	\begin{table*}[h!]
		\renewcommand{\arraystretch}{1.5}
		\centering
		\begin{tabular}{c|c}
			$\pi$ & $\phi(\pi)$ \\ \hline
			case (1) & in $\cD_{2n,2k}$, neither $(2k)(2k+1)$ nor $(2k+1)(2k+2)$ is a factor \\ 
			case (2) & in $\cD_{2n,2k}$, $(2k)(2k+1)$ is a factor\\
			case (3-a) & in $\bigcup_{i=k}^{n-1}\cD_{2n-2,2i}$ \\ 
			case (3-b) & in $\cD_{2n,2k}$, $(2k)(2k+1)$ is not a factor while $(2k+1)(2k+2)$ is a factor   
		\end{tabular}
	\end{table*}
\end{proof}

Now that the recurrences \eqref{id:GI even rec} and \eqref{rec:S} have been understood combinatorially via the mappings $\varphi$ and $\phi$, respectively, we are in a position to give a bijective proof of~\eqref{id:GI and S}.

\begin{proof}[A bijective proof of~\eqref{id:GI and S}]
	It suffices to construct a recursively defined bijection $\Phi:=\Phi_{n,k}$ from $\GI_{2n,2k}$ to $\DIII_{2n,2k}$. More precisely, for a fixed pair $(n,k)$, assuming that $\Phi_{n',k'}$ has been defined for all pairs $(n',k')$ such that either $n'<n$ or $n'=n, k'<k$, then we define $\Phi_{n,k}$ to be the unique mapping that makes the two diagrams in Fig.~\ref{Rbij} commutative. 
	\begin{figure}[h]
		\begin{tikzpicture}[scale=1]
			\draw(0,0) node{$\GI_{2n,2k}$};
			\draw(3,0) node{$\D^{\mathrm{III}}_{2n,2k}$};
			\draw(0,3) node{$\GI_{2n,2k-2}$};
			\draw(3,3) node{$\D^{\mathrm{III}}_{2n,2k-2}$};
			
			\draw(3.3,1.5) node{$\phi$};
			\draw(1.5,3.3) node{$\Phi_{n,k-1}$};
			\draw(1.5,0.3) node{$\Phi_{n,k}$};
			\draw(-.3,1.5) node{$\varphi$};
			
			\draw[>=triangle 45, ->, dashed] (0.9,0) -- (2.1,0);
			\draw[>=triangle 45, ->] (0.9,3) -- (2.1,3);
			\draw[>=triangle 45, ->] (0,0.5) -- (0,2.5);
			\draw[>=triangle 45, ->] (3,0.5) -- (3,2.5);
			
			\draw(5.5,1.5) node{and for $i \geq k-1$,};
			\draw(8,0) node{$\GI_{2n,2k}$};
			\draw(11,0) node{$\D^{\mathrm{III}}_{2n,2k}$};
			\draw(8,3) node{$\GI_{2n-2,2i}$};
			\draw(11,3) node{$\D^{\mathrm{III}}_{2n-2,2i}$};
			
			\draw(11.3,1.5) node{$\phi$};
			\draw(9.5,3.3) node{$\Phi_{n-1,i}$};
			\draw(9.5,0.3) node{$\Phi_{n,k}$};
			\draw(7.7,1.5) node{$\varphi$};
			
			\draw[>=triangle 45, ->, dashed] (8.9,0) -- (10.1,0);
			\draw[>=triangle 45, ->] (8.9,3) -- (10.1,3);
			\draw[>=triangle 45, ->] (8,0.5) -- (8,2.5);
			\draw[>=triangle 45, ->] (11,0.5) -- (11,2.5);
			
			\draw(1.5,1.5) node[rotate=180, scale=2, transform shape] {$\circlearrowright$};
			
			\draw(9.5,1.5) node[rotate=180, scale=2, transform shape] {$\circlearrowright$};
		\end{tikzpicture}
		\caption{The recursively defined bijection $\Phi_k$ from $\GI_{2n,2k}$ to $\D^{\mathrm{III}}_{2n,2k}$}
		\label{Rbij}
	\end{figure}
	
	In other words, for any permutation $\pi\in\GI_{2n,2k}$, we define
	$$\Phi(\pi) := \phi^{-1}\circ\Phi\circ\varphi(\pi)\in\DIII_{2n,2k},$$
	where the inner mapping $\Phi$ could be either $\Phi_{n,k-1}$ (as in the left diagram of Fig.~\ref{Rbij}), or $\Phi_{n-1,i}$ for a certain $k-1\le i\le n-1$ (as in the right diagram of Fig.~\ref{Rbij}), depending on whether the image $\varphi(\pi)$ is in $\GI_{2n,2k-2}$ or in $\bigcup_{i=k-1}^{n-1}\GI_{2n-2,2i}$.
	 
	Initially for the case $n=k=1$, there is only one permutation $21$ in $\GI_{2,2}$, and only one permutation $12$ in $\DIII_{2,2}$, so we simply set $\Phi_{1,1}(21)=12$. The proof is now completed by double induction on $n$ and $k$.
\end{proof}

We list below the one-to-one correspondences between $\GI_{6,6}$ and $\DIII_{6,6}$, and between $\GI_{6,4}$ and $\DIII_{6,4}$, which are determined by the bijection $\Phi$.
\begin{table}[H]
	\renewcommand{\arraystretch}{1.4}
	\centering
	\begin{tabular}{c|c||c|c}
		$\GI_{6,6}$ & $\DIII_{6,6}$ & $\GI_{6,4}$ & $\DIII_{6,4}$\\ \hline
		651432 & 136425 & 431652 & 134256 \\
		652143 & 164235 & 432165 & 142356 \\
		653142 & 135624 & 436512 & 134562 \\
		653214 & 156234 & 436521 & 145623 \\
		653412 & 136245 & 421653 & 146235 \\ 
		653421 & 162345 & 416532 & 134625 \\ 
		654312 & 135642 & & \\
		654321 & 156423 & & \\
	\end{tabular}
	\caption{The correspondences $\GI_{6,6}\xrightarrow{\Phi}\DIII_{6,6}$ and $\GI_{6,4}\xrightarrow{\Phi}\DIII_{6,4}$}
\end{table}

\subsection{A bijection between \texorpdfstring{$\CO_{2n+2}$}{CO2n+2} and \texorpdfstring{$\fD_{2n}$}{D2n}}

Our goal in this and next subsections is to construct two bijections $\theta$ and $\vartheta$, and then combine them to get a bijection $\Theta$ between $\CO_{2n+2}$ and $\GI_{2n}$ that completes the following triangle. This is a bijective approach to \eqref{id:Co and GI}.

\begin{figure}[h]
		\begin{tikzpicture}[scale=0.55]
			\draw(0,0) node{$\CO_{2n+2}$};
			\draw(7,0) node{$\GI_{2n}$};
			\draw(3.5,6) node{$\fD_{2n}$};
			
			\draw(3.5,0.5) node{$\Theta=\vartheta^{-1}\circ\theta$};
			\draw(1.2,3) node{$\theta$};
			\draw(5.8,3) node{$\vartheta$};
			
			\draw[>=triangle 45, ->, dashed] (1.1,0) -- (6.1,0);
			\draw[>=triangle 45, ->] (0,0.6) -- (3.1,5.5);
			\draw[>=triangle 45, ->] (7,0.6) -- (3.9,5.5);
		\end{tikzpicture}
		\caption{The composed bijection $\Theta=\vartheta^{-1}\circ\theta: \CO_{2n+2}\to\GI_{2n}$}
		\label{triangle}
	\end{figure}

\begin{theorem}\label{bijection between D and Co}
For every $n\ge 1$, there exists a bijection
\begin{align*}
	\theta: \CO_{2n+2} &\rightarrow \fD_{2n}\\
	\pi=\pi_1\pi_2 \cdots \pi_{2n+1}\pi_{2n+2} &\mapsto \sigma=\sigma_1\cdots\sigma_{2n},
\end{align*}
where for $1\le i\le n$, $\sigma_{2i-1}:=\pi_{n+i+1}-1$, and $\sigma_{2i}:=\pi_{i+1}-1$.
\end{theorem}
\begin{proof}
First note that due to the constraints shown in \eqref{defn:collapse}, any collapsed permutation $\pi\in\CO_{2n+2}$ must begin with $1$ and end with $2n+2$, and its image $\theta(\pi)=\sigma$ includes each entry from $\{\pi_2-1,\pi_3-1,\ldots,\pi_{2n}-1,\pi_{2n+1}-1\}=[2n]$ exactly once, so $\sigma$ is a permutation of $[2n]$. It remains to check that it belongs to $\fD_{2n}$.

Indeed, applying the condition \eqref{defn:collapse} for a permutation $\pi\in\CO_{2n+2}$, we see that for every $1\le i\le n$, there holds
\begin{align}
	& 1+\lfloor\frac{\sigma_{2i-1}+1}{2}\rfloor\le n+i+1\le n+1+\lfloor\frac{\sigma_{2i-1}+1}{2}\rfloor, \label{ineq:theta-odd}\\
	& 1+\lfloor\frac{\sigma_{2i}+1}{2}\rfloor\le i+1\le n+1+\lfloor\frac{\sigma_{2i}+1}{2}\rfloor.\label{ineq:theta-even}
\end{align}
The left inequality in \eqref{ineq:theta-odd} always holds while the right one is equivalent to
\begin{align}
	\sigma_{2i-1}\ge 2i-1.\label{D-perm-odd}
\end{align}
In the same vein, the right inequality in \eqref{ineq:theta-even} always holds while the left one is equivalent to
\begin{align}
	\sigma_{2i}\le 2i.\label{D-perm-even}
\end{align}
Now \eqref{D-perm-odd} and \eqref{D-perm-even} indicate precisely that $\sigma\in\fD_{2n}$. It is easy to see how $\theta^{-1}$ should be defined and thus the mapping $\theta$ is invertible and bijective, as claimed.

	
\end{proof}

\subsection{A bijection between \texorpdfstring{$\GI_{2n}$}{G2n} and \texorpdfstring{$\mathfrak{D}_{2n}$}{D2n} }

Recall the set of D-permutations $\fD_{2n}$ and the set of E-permutations $\cE_{2n}$ from Definition~\ref{def:five perm} items (5) and (6).

\begin{Def}[{\cite[Sec.~4.1]{LW22}}]\label{d cycyle def}
	 Let $A$ be a finite subset of $\mathbb{Z}_{>0}$, a \emph{D-cycle} (resp.~\emph{E-cycle}) on $A$ refers to a D-permutation (resp.~E-permutation) with exactly one cycle, denote all D-cycles (resp.~E-cycles) on $A$ by $\DC_{A}$ (resp.~$\EC_A$). We simply write $\DC_n$ and $\EC_{n}$ for $\DC_{[n]}$ and $\EC_{[n]}$, respectively.
\end{Def}

\begin{example}\label{cycle example}
We list below all of the D-cycles and E-cycles on the sets $[4]$ and $[6]$.
	\begin{align*}
		&\DC_{4}=\{(1~3~4~2)\},\quad \EC_{4}=\{(1~2~3~4)\}, \\
		&\DC_{6}=\{(1~3~5~6~4~2),~(1~4~3~5~6~2),~(1~5~6~3~4~2)\}, \\ 
		&\EC_{6}=\{(1~2~3~4~5~6),~(1~2~4~3~5~6),~(1~2~5~6~3~4)\}.
\end{align*}
\end{example}

Lin and Yan extended the notion of cycles to the multiset setting, and they showed \cite[Theorem~3]{LY22} that for any fixed multiset $M$, there exists a bijection $\PLY: \EC_M\to \DC_M $. They utilized this bijection to prove another conjecture of Lazar and Wachs \cite[Conjecture~6.5]{LW22} that implies conjecture~\ref{conj:Lazar-Wachs}. It is worth noting that their bijective proof relying on $\PLY$ appears to be the only bijective proof of this stronger conjecture. Although not explicitly mentioned in \cite{LY22}, we note that the permutation version of their mapping, i.e., $\PLY:\EC_{2n}\to\DC_{2n}$ can be extended to a bijection between the entire sets of E-permutations and D-permutations of length $2n$. This in turn will produce a bijection from $\GI_{2n}$ to $\fD_{2n}$ when our variant of Foata's first fundamental transformation $\Psi$ is involved. That is to say, the bijection we construct below is the composed mapping
$$\vartheta:=\PLY\circ\Psi^{-1}.$$

We need a few more definitions before we can present this bijection.

\begin{Def}\label{even cyclic double ascent def}
Given a cycle $(c_1 \ldots c_k)$, an element $c_i$ is called an \emph{ even cyclic double ascent} if $c_{i-1}<c_i<c_{i+1}$ (with the convention that $c_{k+1}:=c_1$), and $c_i$ is even. The smallest number in each cycle is called a \emph{cycle minimum}. Given a permutation $\pi$, let $\Cmin(\pi)$ denote the set of cycle minima of $\pi$, and let $\Lrmin(\pi)$ denote the set of \emph{left-to-right minima} of $\pi$.
\end{Def}

We shall give a direct description of $\vartheta$ sending a permutation $\pi\in\GI_{2n}$ to a permutation $\sigma\in\fD_{2n}$. It consists of two steps. Note that step (1) is to apply $\Psi^{-1}$ on $\pi$, while step (2) is to further apply $\PLY$.
\begin{enumerate}
	\item Separate the one-line notation of $\pi$ into cycles according to its left-to-right minima (i.e., insert the left parenthesis ``('' before every left-to-right minimum of $\pi$). Reverse the order inside each cycle (i.e., $(c_1~c_2\cdots c_k)$ becomes $(c_1~c_k~c_{k-1}\cdots c_2)$). Let the resulting expression be the cycle notation of a permutation $\hat{\pi}$.
	\item Inside each cycle, say $(c_1 \cdots c_k)$ of $\hat{\pi}$, rightward shift each even cyclic double ascent, say $c_i$, until it is larger than its successor\footnote{Since we are shifting inside a cycle, every letter has a successor. In the extreme case senario, $c_i$ will be shifted to the right end of the cycle since it is now larger its successor, namely the left end as well as the minimal letter in this cycle. So this shifting operation is well-defined.}. Let this final expression be the cycle notation of the image permutation $\sigma$.
\end{enumerate}



\begin{theorem}\label{linyanbijection}
	The mapping $\vartheta$ as defined above is a bijection from $\GI_{2n}$ to $\fD_{2n}$, such that for any $\pi\in\GI_{2n}$ and $\sigma:=\vartheta(\pi)\in\fD_{2n}$, we have
	\begin{align}\label{GI to D-stat}
	\Lrmin(\pi) = \Cmin(\sigma).
	\end{align}  
\end{theorem}
\begin{proof}
We first show that when $\vartheta$ is applied on a permutation $\pi\in\GI_{2n}$, its image $\sigma$ is indeed in $\fD_{2n}$. Since both $\sigma_{2i}=2i$ and $\sigma_{2i-1}=2i-1$ are allowed in a D-permutation, we can ignore all the fixed points (i.e., $1$-cycles) of $\sigma$. The following observations are sufficient. 
\begin{enumerate}[label=(\roman*)]
	\item The set of left-to-right minima of $\pi$ becomes the set of cycle minima of $\hat{\pi}$, and the shifting operation in step (2) preserves cycle minima so $\Cmin(\hat{\pi})=\Cmin(\sigma)$. In particular, the cycle minima of non singleton cycles are all odd numbers since an even number cannot be an ascent bottom in $\pi$.
	\item Inside each non singleton cycle of $\hat{\pi}$, every even number is larger than its predecessor, every odd number is smaller than its successor.
	\item For any non fixed point $i$ in $\sigma$, if $i$ is odd, then $\sigma_i>i$ since $i$ is smaller than its successor after step (1), and no even number could become its new successor after the shifting operation in step (2); if $i$ is even then $\sigma_i<i$, since either $i$ is larger than its successor after step (1) and no bigger number get passed it in step (2), or it is smaller than its successor after step (1), and in view of (ii) we see that it is an even cyclic double ascent and will be shifted in step (2), leading to $\sigma_i<i$ as well.
\end{enumerate}
All of the above claims can be verified using the $\eE,\eO,\oO$ avoiding condition on $\pi\in\GI_{2n}$, the details are omitted. Observation (iii) implies that $\sigma\in\fD_{2n}$. Both steps (1) and (2) are invertible so $\vartheta$ is a bijection. Finally, observation (i) implies precisely \eqref{GI to D-stat}.
\end{proof}



\begin{example}
Given a permutation $\pi=5~10~7~12~11~9~8~4~3~1~6~2\in\GI_{12}$, we show the steps in deriving its image $\sigma$ under $\vartheta$.
	\begin{align*}
		\pi = 5~10~7~12~11~9~8~4~3~1~6~2\quad &\xrightarrow{\text{split into cycles}}\quad (5~10~7~12~11~9~8)(4)(3)(1~6~2) \\
		&\xrightarrow{\text{reverse order}}\quad (5~8~9~11~12~7~10)(4)(3)(1~2~6) \\
		&\xrightarrow{\text{step 2}}\quad (5~9~11~12~8~7~10)(4)(3)(1~6~2) = \sigma\in\fD_{12}.
	\end{align*}
\end{example}


\section{Concluding remarks}
It is worth noting that the recurrence relation shown in \eqref{id:GI even rec} fully characterizes the entries in the even-indexed columns of the triangle in Table~\ref{Tab:GI}. There are further properties of this triangle that are worth exploring. Moreover, refining the counting of permutations in $\GII$, $\GIII$, and $\GIV$ according to their first letters gives rise to three more Seidel-like triangles comparable to Table~\ref{Tab:GI}. 

Lastly, it is well known that the $n$-th Genocchi median $h_n$ is divisible by $2^n$ (see \cite{BD81,big14,PZ23}), and the sequence $\{h_n/2^n\}_{n\ge 0}=\{1,1,2,7,38,295,\ldots\}$ is registered as A000366 in the OEIS~\cite{slo} and known as the \emph{normalized median Genocchi numbers}. Given any combinatorial model enumerated by the median Genocchi numbers, it is then natural to wonder if this model allows a manifest explanation to the aforementioned divisibility of $h_n$ by $2^n$; see \cite[Remark~4.4]{AHT22}. Initial computations suggest that the four permutation models introduced in this paper, namely $\GI$, $\GII$, $\GIII$, and $\GIV$ all furnish such an explanation. We plan to continue these investigations in our ensuing work~\cite{FFLY26}.

\section*{Acknowledgement}
Shishuo Fu was partially supported by the National Natural Science Foundation of China grants 12171059. The authors would also like to acknowledge the support from the Mathematical Research Center of Chongqing University.

\end{document}